\documentclass[11pt,reqno]{amsart}
\usepackage[a4paper, hmargin={2.7cm,2.7cm},vmargin={3.3cm,3.3cm}]{geometry}
\usepackage[english]{babel}
\usepackage{color}
\usepackage{amsmath}
\usepackage{amssymb}
\usepackage{amsfonts}
\usepackage{amsthm}
\usepackage{mathrsfs}
\usepackage{amsthm}
\usepackage[noadjust]{cite}
\usepackage{dsfont}
\usepackage{etoolbox}
\usepackage[unicode=true]{hyperref}
\usepackage[noabbrev]{cleveref}
\usepackage{esint}
\usepackage{todonotes}
\usepackage{crossreftools}
\usepackage{csquotes}

\MakeOuterQuote{"}

\usepackage{xparse}
\let\realItem\item
\makeatletter
\NewDocumentCommand\myItem{ o }{%
   \IfNoValueTF{#1}%
      {\realItem}
      {\realItem[#1]\def\@currentlabel{#1}}
}
\makeatother

\usepackage{enumitem}    
\setlist[enumerate]{
  topsep=1pt,
  leftmargin=30pt,
  before=\let\item\myItem,
  label=\textnormal{(\arabic*)},
  widest=(a2')
}

\binoppenalty=\maxdimen
\relpenalty=\maxdimen
\makeatletter
\patchcmd{\ttlh@hang}{\parindent\z@}{\parindent\z@\leavevmode}{}{}
\patchcmd{\ttlh@hang}{\noindent}{}{}{}
\makeatother

\makeatletter
\@namedef{subjclassname@2020}{\textup{2020} Mathematics Subject Classification}
\makeatother

\theoremstyle{plain}
\newtheorem{theorem}{Theorem}[section]
\newtheorem{lemma}[theorem]{Lemma}
\newtheorem{proposition}[theorem]{Proposition}

\newtheorem{innercustomgeneric}{\customgenericname}
\providecommand{\customgenericname}{}
\newcommand{\newcustomtheorem}[2]{%
  \newenvironment{#1}[1]
  {%
   \renewcommand\customgenericname{#2}%
   \renewcommand\theinnercustomgeneric{##1}%
   \innercustomgeneric
  }
  {\endinnercustomgeneric}
}

\def\Xint#1{\mathchoice
{\XXint\displaystyle\textstyle{#1}}%
{\XXint\textstyle\scriptstyle{#1}}%
{\XXint\scriptstyle\scriptscriptstyle{#1}}%
{\XXint\scriptscriptstyle\scriptscriptstyle{#1}}%
\!\int}
\def\XXint#1#2#3{{\setbox0=\hbox{$#1{#2#3}{\int}$ }
\vcenter{\hbox{$#2#3$ }}\kern-.6\wd0}}

\def\dashint{\Xint-}

\newcustomtheorem{customthm}{Theorem}
\newcustomtheorem{customlemma}{Lemma}

\theoremstyle{definition}

\theoremstyle{remark}
\newtheorem{remark}[theorem]{Remark}
\newtheorem*{remark*}{Remark}

\numberwithin{equation}{section}

\makeatletter
\providecommand*{\cupdot}{%
  \mathbin{%
    \mathpalette\@cupdot{}%
  }%
}
\newcommand*{\@cupdot}[2]{%
  \ooalign{%
    $\m@th#1\cup$\cr
    \sbox0{$#1\cup$}%
    \dimen@=\ht0 %
    \sbox0{$\m@th#1\cdot$}%
    \advance\dimen@ by -\ht0 %
    \dimen@=.5\dimen@
    \hidewidth\raise\dimen@\box0\hidewidth
  }%
}

\providecommand*{\bigcupdot}{%
  \mathop{%
    \vphantom{\bigcup}%
    \mathpalette\@bigcupdot{}%
  }%
}
\newcommand*{\@bigcupdot}[2]{%
  \ooalign{%
    $\m@th#1\bigcup$\cr
    \sbox0{$#1\bigcup$}%
    \dimen@=\ht0 %
    \advance\dimen@ by -\dp0 %
    \sbox0{\scalebox{2}{$\m@th#1\cdot$}}%
    \advance\dimen@ by -\ht0 %
    \dimen@=.5\dimen@
    \hidewidth\raise\dimen@\box0\hidewidth
  }%
}
\makeatother

\DeclareMathOperator*{\diag}{diag}

\DeclareMathOperator*{\scale}{scale}

\DeclareMathOperator*{\diam}{diam}

\usepackage{xparse}

\newcommand{\Indicator}{\mathds{1}}

\newcommand{\BS}{\dot{\mathbf{b}}^{\alpha}_{p,q}}
\newcommand{\TLone}{\dot{\mathbf{f}}^{\alpha_1}_{p_1,q_1}}
\newcommand{\TLtwo}{\dot{\mathbf{f}}^{\alpha_2}_{p_2,q_2}}
\newcommand{\TLonep}{\dot{\mathbf{f}}^{\alpha_1}_{p,q_1}}
\newcommand{\TLtwop}{\dot{\mathbf{f}}^{\alpha_2}_{p,q_2}}
\newcommand{\TLnaked}{\dot{\mathbf{f}}}
\newcommand{\TL}{\dot{\mathbf{f}}^{\alpha}_{p,q}}
\NewDocumentCommand\TLseq{O{\alpha}}{\mathbf{f}^{#1}_{p,q}}
\newcommand{\TLA}{\dot{\mathbf{f}}^{\alpha}_{p,q}(A)}
\newcommand{\TLB}{\dot{\mathbf{f}}^{\alpha}_{p,q}(B)}
\newcommand{\TLAi}{\dot{\mathbf{f}}^{\alpha}_{\infty,q}(A)}
\newcommand{\TLAii}{\dot{\mathbf{f}}^{\alpha}_{\infty,\infty}(A)}

\newcommand{\vertiii}[1]{{\left\vert \kern-0.25ex
                            \left\vert \kern-0.25ex
                              \left\vert #1\right\vert\kern-0.25ex
                            \right\vert \kern-0.25ex
                          \right\vert}}

\newcommand{\eps}{\varepsilon}
\renewcommand{\emptyset}{\varnothing}

\newcommand{\GL}{\operatorname{GL}}

\renewcommand{\Im}{\operatorname{Im}}

\DeclareFontFamily{U}{mathx}{\hyphenchar\font45}
\DeclareFontShape{U}{mathx}{m}{n}{
      <5> <6> <7> <8> <9> <10>
      <10.95> <12> <14.4> <17.28> <20.74> <24.88>
      mathx10
      }{}
\DeclareSymbolFont{mathx}{U}{mathx}{m}{n}
\DeclareFontSubstitution{U}{mathx}{m}{n}
\DeclareMathAccent{\widecheck}{0}{mathx}{"71}
\DeclareMathAccent{\wideparen}{0}{mathx}{"75}

\newcommand{\Q}{\mathbb{Q}}
\newcommand{\R}{\mathbb{R}}

\newcommand{\CC}{\mathbb{C}}
\newcommand{\N}{\mathbb{N}}
\newcommand{\Z}{\mathbb{Z}}

\title[Discrete Triebel-Lizorkin spaces and expansive matrices]{Discrete Triebel-Lizorkin spaces and expansive matrices}

\author{Jordy Timo van Velthoven}

\address[JvV]{Faculty of Mathematics,
University of Vienna,
Oskar-Morgenstern-Platz 1,
A-1090 Vienna, Austria}

\email{jordy.timo.van.velthoven@univie.ac.at}

\author{Felix Voigtlaender}

\address[FV]{Mathematical Institute for Machine Learning and Data Science (MIDS),
Catholic University of Eichstätt–Ingolstadt (KU),
Auf der Schanz 49, 85049 Ingolstadt, Germany}

\email{felix.voigtlaender@ku.de}



\begin{document}

\maketitle

\begin{abstract}
We provide a characterization of two expansive dilation matrices
yielding equal discrete anisotropic Triebel-Lizorkin spaces.
For two such matrices $A$ and $B$, and arbitrary $\alpha \in \R$ and $p,q \in (0,\infty]$,
it is shown that $\TL(A) = \TL(B)$ if and only if the set $\{A^j B^{-j} : j \in \Z\}$ is finite,
or in the trivial case when $|\det(A)|^{\alpha + 1/2 - 1/p} = |\det(B)|^{\alpha + 1/2 - 1/p}$ and $p = q$.
This provides an extension of a result by Triebel for diagonal dilations to arbitrary expansive matrices.
The obtained classification of dilations is different from corresponding results for anisotropic Triebel-Lizorkin function spaces.
\end{abstract}

\section{Introduction}

Let $A \in \mathrm{GL}(d, \mathbb{R})$ be an expansive matrix, i.e., all eigenvalues
$\lambda \in \CC$ of $A$ satisfy $|\lambda|>1$.
The associated discrete (homogeneous) Besov space $\BS(A)$ and Triebel-Lizorkin space $\TL(A)$
are defined to consist of those sequences $c \in \CC^{\Z \times \Z^d}$ satisfying
\[
  \| c \|_{\BS(A)}
  := \bigg(
       \sum_{j \in \Z}
         |\det(A)|^{-jq (\alpha + 1/2)}
         \bigg\| \sum_{k \in \Z^d} |c_{j,k}| \mathds{1}_{A^j ([0, 1]^d + k)} \bigg\|_{L^p}^q
     \bigg)^{1/q}
   < \infty
\]
when $\alpha \in \R$ and $p, q \in (0, \infty]$ (with the usual modification for $q = \infty$), respectively
\begin{align} \label{eq:TL_intro}
  \| c \|_{\TLA}
  := \bigg\|
       \bigg(
         \sum_{j \in \mathbb{Z}}
           |\det(A)|^{-jq( \alpha + 1/2) }
           \sum_{k \in   \mathbb{Z}^d}
             |c_{j,k}|^q \Indicator_{A^j ([0,1]^d + k)}
       \bigg)^{1/q}
     \bigg\|_{L^p} < \infty,
\end{align}
when $\alpha \in \R$, $p \in (0, \infty)$ and $q \in (0, \infty]$
(with the usual modifications when $q = \infty$);
see Section \ref{sec:TL} for the more technical definition of the spaces $\TLA$ when $p = \infty$.
The importance of these sequence spaces is that they characterize the decay of wavelet coefficients
of functions/distributions in the associated Besov and Triebel-Lizorkin function spaces.
As such, many problems regarding Besov and Triebel-Lizorkin spaces can be reduced
to the corresponding sequence spaces, which are often easier to work with;
see, e.g., \cite{bownik2008duality, bownik2006atomic, frazier1991littlewood, frazier1990discrete}.
Moreover, discrete Triebel-Lizorkin spaces (also called discrete Littlewood-Paley spaces)
naturally occur in the study of weighted inequalities and Carleson coefficients,
see, e.g., \cite{hanninen2018equivalence, verbitsky1996imbedding, hanninen2020two}
and the references therein.

One question on anisotropic Besov and Triebel-Lizorkin function spaces that has received particular
attention over the last years is how they depend on the choice of the dilation matrix;
see \cite{cheshmavar2020classification, koppensteiner2023classification, bownik2003anisotropic, velthoven2024classification}.
The same question for their corresponding sequence spaces appears to have been first considered
in \cite{triebel2004wavelet} and was motivated by the so-called \emph{transference method},
which allows to transfer problems for anisotropic Besov spaces via sequence spaces
to \emph{isotropic} Besov spaces.
Indeed, note that it follows readily from the definition of discrete Besov spaces that
$\dot{\mathbf{b}}_{p,q}^{\alpha} (A) = \dot{\mathbf{b}}_{p,q}^{\alpha} (B)$
for two expansive matrices $A, B \in \mathrm{GL}(d, \mathbb{R})$ if and only if
\begin{align} \label{eq:det_intro}
  |\det(A)|^{\alpha + 1/2 - 1/p} =  |\det(B)|^{\alpha + 1/2 - 1/p}.
\end{align}
In other words, the scale of discrete anisotropic Besov spaces is \emph{independent} of the choice of the dilation matrix.
Similarly, it follows readily from the (quasi-)norms  \eqref{eq:TL_intro} that if $p = q$
and the identity \eqref{eq:det_intro} holds, then $\TLnaked^{\alpha}_{p,q} (A) = \TLnaked^{\alpha}_{p,q} (B)$.
However, the question whether the full scale of discrete Triebel-Lizorkin spaces is independent
of the choice of the dilation matrix is remarkably more subtle.
As a matter of fact, it was conjectured in \cite[Conjecture 11]{triebel2004wavelet}
that for diagonal matrices $A = \diag(2^{a_1}, \dots, 2^{a_d})$ and $B = \diag(2^{b_1}, \dots, 2^{b_d})$
with anisotropies $ (a_1, \dots, a_d), (b_1, \dots, b_d) \in (0, \infty)^d$
satisfying $\sum_{i = 1}^d a_i = \sum_{i = 1}^d b_i = d$ (so that $\det(A) = \det(B)$),
the spaces $\TL(A) $ and $ \TL(B)$ coincide for all $\alpha \in \R$, $p \in (0, \infty)$
and $q \in (0, \infty]$.
The fact that the conjecture \cite[Conjecture 11]{triebel2004wavelet}
was incorrect as stated was shown by the following theorem; see \cite[Proposition 5.26]{triebel2006theory}.

\begin{theorem}[\cite{triebel2006theory}] \label{thm:triebel}
Let $\alpha \in \R$, $p \in (0, \infty)$ and $q \in (0, \infty]$.
Let $A = \diag(2^{a_1}, \dots, 2^{a_d})$ and $ B = \diag(2^{b_1}, \dots, 2^{b_d})$ for anisotropies
\[
  (a_1, \dots, a_d), (b_1, \dots, b_d) \in (0, \infty)^d
  \quad \text{satisfying} \quad
  \sum_{i = 1}^d a_i = \sum_{i = 1}^d b_i = d.
\]
Suppose that $A \neq B$. Then $\TL(A) = \TL(B)$ if and only if $p = q$.
\end{theorem}

The aim of the present paper is to provide a characterization of two arbitrary expansive matrices
$A, B \in \mathrm{GL}(d, \R)$ yielding the same discrete Triebel-Lizorkin spaces $\TL(A) = \TL(B)$.
The following theorem provides our sufficient condition for the coincidence of sequence spaces.

\begin{theorem} \label{thm:sufficient_intro}
If $A, B \in \mathrm{GL}(d, \R)$ are expansive matrices such that $\{A^j B^{-j} : j \in \Z\}$ is a finite set,
then $\TL(A) = \TL(B)$ for all $\alpha \in \R$ and $p, q \in (0, \infty]$.
\end{theorem}

We mention that the set $\{A^j B^{-j} : j \in \Z\}$ is finite if and only if $A^k = B^k$
for some $k \in \mathbb{N}$, cf.\ \Cref{lem:characterization}.
It is of interest to compare this sufficient condition to the characterization
of two expansive matrices yielding equal anisotropic (homogeneous) Triebel-Lizorkin function spaces.
In \cite{koppensteiner2023classification}, it was shown that such spaces coincide
if and only if we have $p \in (1, \infty)$, $q = 2$ and $\alpha = 0$, or if
\begin{align} \label{eq:equivalent_matrices}
 \sup_{j \in \Z} \| A^{-j} B^{ \lfloor \varepsilon j \rfloor} \| < \infty,
\end{align}
with $\varepsilon := \ln |\det(A)| / \ln |\det(B)|$, which in turn is equivalent
to two quasi-norms associated to $A$ and $B$ being equivalent; see \cite[Section 10]{bownik2003anisotropic}.
The condition that $\{A^j B^{-j} : j \in \Z\}$ is a finite set (cf.\ \Cref{thm:sufficient_intro}) implies,
in particular, that $|\det(A)| = |\det(B)|$ and is thus much stronger than the condition \eqref{eq:equivalent_matrices}
classifying anisotropic Triebel-Lizorkin function spaces \cite{koppensteiner2023classification}.

We make some further comments on the necessity of the sufficient condition of \Cref{thm:sufficient_intro}.
First, we mention that if two expansive matrices $A, B \in \mathrm{GL}(d, \R)$
have only positive eigenvalues and satisfy $\det(A) = \det(B)$ and \eqref{eq:equivalent_matrices},
then necessarily $A = B$, cf.\ \cite[Theorem 7.9]{cheshmavar2020classification}.
In particular, if $A, B \in \mathrm{GL}(d, \R)$ are expansive matrices having only positive eigenvalues
and satisfy the sufficient condition of \Cref{thm:sufficient_intro}, then $A = B$.
Second, it is not difficult to show that the matrices $A = \diag(2, 2)$ and $B = \diag(2, -2)$
provide examples of matrices $A \neq B$ still yielding $\TL(A) = \TL(B)$.
Thus, in contrast to the setting of \Cref{thm:triebel}, the coincidence of spaces $\TL(A) = \TL(B)$
is generally \emph{not} equivalent to $A= B$.
Lastly, we mention that for the matrices $A = 2 \cdot I$ and $B = 2 \cdot R_{\phi}$, where
\[
 R_{\phi} =
 \begin{pmatrix}
  \cos(\phi) & - \sin(\phi) \\
  \sin(\phi) & \cos(\phi)
 \end{pmatrix}
 , \quad \phi \in \mathbb{R} \setminus \mathbb{Q},
\]
the spaces $\TL(A)$ and $\TL(B)$ can be shown to be distinct in case $p \neq q$ (see also \Cref{thm:necessary_intro}).
Note that for such matrices the set $\{ A^j B^{-j} : j \in \Z\}$ is infinite,
which follows, for example, from Weyl's equidistribution theorem.

Our main result shows that the sufficient condition provided by \Cref{thm:sufficient_intro}
is in general also necessary for the coincidence of the scale of discrete Triebel-Lizorkin spaces.
More precisely, we show the following general theorem.

\begin{theorem} \label{thm:necessary_intro}
Let $A, B \in \mathrm{GL}(d, \R)$ be expansive, $\alpha_1, \alpha_2 \in \R$ and $p_1, p_2, q_1, q_2 \in (0, \infty]$.

If $\TLone(A) = \TLtwo(B)$, then $ p_1 = p_2$ and at least least one of the following conditions hold:
\begin{enumerate}
 \item[(i)] The set $\{A^j B^{-j} : j \in \Z\}$ is finite, $\alpha_1 = \alpha_2$ and $q_1 = q_2$;
 \item[(ii)] $p_1 = p_2 = q_1 = q_2$ and $|\det(A)|^{\alpha_1 + 1/2 - 1/{p_1}} = |\det(B)|^{\alpha_2 + 1/2 - 1/{p_2}}$.
\end{enumerate}
\end{theorem}

\Cref{thm:sufficient_intro} and \Cref{thm:necessary_intro} provide a full classification
of the expansive dilation matrices yielding equal discrete Triebel-Lizorkin spaces,
and form a nontrivial extension of \Cref{thm:triebel} to arbitrary expansive dilations.
In addition, the necessary conditions for possibly different integrability and smoothness exponents
provided by \Cref{thm:necessary_intro} appear to be new even for diagonal dilation matrices.
Similarly, a classification of dilations for the case $p = \infty$ seems new for diagonal dilations.

As already mentioned above, the classification provided by \Cref{thm:sufficient_intro}
and \Cref{thm:necessary_intro} is different from the one for anisotropic Triebel-Lizorkin function spaces \eqref{eq:equivalent_matrices}.
To illustrate this, we recall that an expansive matrix $A \in \mathrm{GL}(d, \mathbb{R})$
is equivalent to the isotropic dilation matrix $2 \cdot I_d$ in the sense of \eqref{eq:equivalent_matrices}
if and only if $A$ is diagonalizable over $\mathbb{C}$ with all eigenvalues being equal in absolute value,
see, e.g., \cite[Example, p.~7]{bownik2003anisotropic}.
Combined with the classification of Triebel-Lizorkin function spaces \cite{koppensteiner2023classification},
except in the trivial case where $p \in (1,\infty)$, $q = 2$, and $\alpha = 0$,
an expansive $A \in \mathrm{GL}(d, \mathbb{R})$ therefore generates the isotropic Triebel-Lizorkin function space
$\dot{\mathbf{F}}^{\alpha}_{p,q}(A) = \dot{\mathbf{F}}^{\alpha}_{p,q}(2 \cdot I_d)$
if and only if $A$ is diagonalizable over $\mathbb{C}$ with all eigenvalues being equal in absolute value.
In contrast, an expansive matrix $A \in \mathrm{GL}(d, \mathbb{R})$ generates
the classical Triebel-Lizorkin sequence space $\TL(A) = \TL(2 \cdot I_d)$ for $p\neq q$
if and only if $A^k = 2^k \cdot I_d$ for some $k \in \mathbb{N}$.
In turn, this is equivalent to $A$ being diagonalizable over $\mathbb{C}$
and such that each eigenvalue is of the form $2z$ for some $z \in \mathbb{C}$
satisfying $z^k = 1$ for some $k \in \mathbb{N}$.
See \Cref{sec:spectral} for further details.

Our proofs of \Cref{thm:sufficient_intro} and \Cref{thm:necessary_intro} are elementary and essentially self-contained.
The condition that the set $\{A^j B^{-j} : j \in \Z\}$ is finite is equivalent to the set $\{B^j A^{-j} : j \in \Z\}$ being finite.
Using this, the central idea in our proof of \Cref{thm:sufficient_intro} is to partition the integers
$\Z = \bigcupdot_{1 \leq t \leq N} J_t$ into sets $J_t := \{ j \in \Z : B^j A^{-j} = M_t \}$
for matrices $M_t$, $1 \leq t \leq N$, where $N := \# \{ B^{j} A^{-j} : j \in \Z\}$.
This allows us to rewrite the (quasi-)norms \eqref{eq:TL_intro} in such a way
that by means of a change of variable the (quasi-)norm of $\TL(A)$ can be compared to that of $\TL(B)$.
For the case $p = \infty$, we use a characterization of the usual (quasi-)norm via
a local $q$-power function as shown in \cite{bownik2008duality} (see \Cref{lem:borelsets}).
The necessary condition provided by \Cref{thm:necessary_intro} requires significantly more work
than the proof of \Cref{thm:sufficient_intro}.
For proving \Cref{thm:necessary_intro}, we construct sequences $c \in \CC^{\Z \times \Z^d}$
that allow us to compare the (quasi-)norms of $\TLone(A)$ and $\TLtwo(B)$ to that
of some (weighted) $\ell^r$-spaces for suitable $r \in \{p_1, p_2, q_1, q_2\}$.
In combination with the equivalence of the (quasi-)norms of $\TLone(A)$ and $\TLtwo(B)$
this allows us then to show the coincidence of the integrability exponents.
Among these different cases, the proof of $p_1 = p_2 = q_1 = q_2$ (see \Cref{thm:necessary_intro}(ii))
is most difficult as it requires the construction of a sequence whose $\TLone(A)$-norm is comparable
to some (weighted) $\ell^{q_1}$-norm, whereas its $\TLtwo(B)$-norm should be
comparable to some (weighted) $\ell^{p_2}$-norm.
The construction of such sequences are based on some ideas underlying the proof of \Cref{thm:triebel}
as given in \cite{triebel2006theory} and form nontrivial adaptations of those sequences to general expansive matrices.

The organization of the paper is as follows: \Cref{sec:TL} provides basic notation and properties
for discrete Triebel-Lizorkin spaces that will be used throughout the paper.
In \Cref{sec:sufficient}, we provide a proof of \Cref{thm:sufficient_intro}.
The proof of \Cref{thm:necessary_intro} is given in \Cref{sec:necessary} and split into various subresults.
Finally, \Cref{sec:spectral} provides a characterization of expansive matrices $A$
for which $\TLA$ coincides with the isotropic spaces $\TLnaked_{p,q}^\alpha (2 \cdot I_d)$.

\subsection*{Notation}

Unless otherwise noted, $\| \cdot \|$ denotes the usual Euclidean norm on $\R^d$.
For a matrix $A \in \R^{k \times d}$, $\| A \|$ denotes the operator norm of $A$.
The open and closed Euclidean balls with radius $r>0$ and center $x \in \R^d$ are denoted
by $\mathcal{B}_r(x)$ and $\overline{\mathcal{B}}_r(x)$, respectively.
The $r$-neighborhood (resp.\ diameter) of a set $X \subseteq \R^d$ with respect to
the Euclidean distance is denoted by $\mathcal{B}_r(X) = \bigcup_{x \in X} \mathcal{B}_r (x)$ (resp.\ $\diam(X)$).
The standard basis of $\mathbb{C}^{\mathbb{Z} \times \mathbb{Z}^d}$ is denoted by
$(e_{j,k})_{j \in \mathbb{Z}, k \in \mathbb{Z}^d}$ and the Kronecker delta function
$\delta$ is as usual defined by $\delta_{i,i} = 1$ and $\delta_{i,j} = 0$ if $i \neq j$.

The cardinality of a set $M$ is denoted by $\#M$, with $\#M \in \mathbb{N}_0$
for a finite set and $\#M = \infty$ for an infinite set.
The Lebesgue measure of a measurable set $X \subseteq \mathbb{R}^d$ is denoted by $|X|$,
and integration of a measurable function $f : \R^d \to \mathbb{C}$ over $X$
is written as $\int_{X} f(x) \; dx$. For a set $X$ of finite positive measure,
we write $\dashint_X f(x) \; dx := |X|^{-1} \int_X f(x) \; dx$.

Given two functions $f_1, f_2 : X \to [0, \infty)$ on a set $X$,
we write $f_1 \lesssim f_2$ if there exists $C>0$ such that $f_1 (x) \leq C f_2 (x)$ for all $x \in X$.
We use the notation $f_1 \asymp f_2$ whenever $f_1 \lesssim f_2$ and $f_2 \lesssim f_1$.
Subscripted variants such as $f_1 \lesssim_{a,b} f_2$ indicate that the implicit constant depends only on quantities $a,b$.

\section{Discrete Triebel-Lizorkin spaces}
\label{sec:TL}

For an expansive matrix $A \in \mathrm{GL}(d, \mathbb{R})$, we define associated \emph{dilated cubes} by
\[
  Q^A_{j,k} := A^j ([0,1]^d + k), \quad j \in \mathbb{Z}, \; k \in \mathbb{Z}^d.
\]
The \emph{scale} of a dyadic cube $Q^A_{j,k}$ is defined by
$\scale(Q^A_{j,k}) = \scale_A (Q^A_{j,k}) = \log_{|\det(A)|} (|Q^A_{j,k}|)$.
We denote the family of all dyadic cubes associated to $A$ by $\mathcal{Q}^A = \{Q^A_{j,k} : j \in \mathbb{Z}, k \in \mathbb{Z}^d \}$.

For $\alpha \in \mathbb{R}$ and $p,q  \in (0, \infty]$, the (homogeneous) anisotropic discrete
Triebel-Lizorkin space $\TLA$ is defined as the space of all sequences
$c \in \mathbb{C}^{\mathbb{Z} \times \mathbb{Z}^d}$ satisfying $\| c \|_{\TLA} < \infty$, where
\[
  \| c \|_{\TLA}
  := \bigg\|
       \bigg(
         \sum_{j \in \mathbb{Z}}
           \sum_{k \in   \mathbb{Z}^d}
             \big(|\det(A)|^{-j( \alpha + 1/2) } |c_{j,k}| \Indicator_{Q_{j,k}^A} \big)^q
       \bigg)^{\frac{1}{q}}
     \bigg\|_{L^p}
\]
if $p < \infty$ (with the usual modification for $q = \infty$), and
\begin{align} \label{eq:TLAi}
  \| c \|_{\TLAi}
  := \sup_{P \in \mathcal{Q}^A}
       \bigg(\;
         \dashint_{P}
           \sum_{\substack{j \in \mathbb{Z} \\ j \leq \scale(P)}}
             \sum_{ k \in \mathbb{Z}^d}
               \big(|\det(A)|^{-j(\alpha + 1/2)} |c_{j,k}| \Indicator_{Q^A_{j,k}} (x) \big)^q
         \; dx
       \bigg)^{\frac{1}{q}},
\end{align}
where the case $q = \infty$ in \eqref{eq:TLAi} has to be interpreted as
\[
  \| c \|_{\TLAii} := \sup_{j \in \mathbb{Z}, k \in \mathbb{Z}^d} |\det (A)|^{-j(\alpha + 1/2)} |c_{j,k}|;
\]
see \cite{bownik2006atomic, bownik2007anisotropic, bownik2008duality} for various basic properties.

In order to give similar proofs for the cases $p < \infty$ and $p = \infty$,
we will often use the following equivalent (quasi-)norms.
The lemma is a direct consequence of a characterization of $\TLAi$ in terms of a so-called local $q$-power function.
See \cite[Corollary 3.4]{bownik2008duality} for a proof.

\begin{lemma}[\cite{bownik2008duality}] \label{lem:borelsets}
Let $\alpha \in \R$ and $p, q \in (0, \infty]$.
Fix $0 < \varepsilon <1$.
Then
\(
  \| c \|_{\TLA} \asymp \| c \|_{\TLA}^\ast
\)
for any sequence $c \in \CC^{\Z \times \Z^d}$, where
\[
  \| c \|_{\TLA}^\ast
  := \inf
     \bigg\{
       \bigg\|
         \bigg(
           \sum_{j \in \Z}
             \sum_{k \in \Z^d}
               \big( |\det(A)|^{-j(\alpha + 1/2)} |c_{j,k}| \Indicator_{E_{j,k}} \big)^q
         \bigg)^{\frac{1}{q}}
       \bigg \|_{L^p}
       :
       E_{j,k} \subseteq Q^A_{j,k} , \; \frac{|E_{j,k}|}{|Q^A_{j,k}|} > \varepsilon
     \bigg\}
\]
where $E_{j,k} \subseteq Q^A_{j,k}$ are Borel sets, with the usual modification for $q = \infty$.
The implicit constant is independent of the sequence $c$.
\end{lemma}

\begin{remark}
 Strictly speaking, the statement of \cite[Corollary 3.4]{bownik2008duality} provides
 a (quasi-)norm characterization in terms of the $L^2$-normalized indicators
 $|E_{j,k}|^{-1/2} \Indicator_{E_{j,k}}$ rather than the functions
 $|Q_{j,k}^A|^{-1/2} \Indicator_{E_{j,k}}$ appearing in the statement of \Cref{lem:borelsets}.
 However, by using that $\varepsilon < |E_{j,k}|/|Q^A_{j,k}| \leq 1$,
 the (quasi-)norm characterization provided by \Cref{lem:borelsets} is easily seen
 to be equivalent to \cite[Corollary 3.4]{bownik2008duality}.
\end{remark}

The following basic properties of discrete Triebel-Lizorkin spaces appear to be well-known.
However, as we will use both properties, but could not locate a proof,
we provide short arguments in the appendix.

\begin{lemma} \label{lem:basic}
Let $A \in \GL(d, \R)$ be expansive, $\alpha \in \R$ and $p,q\in (0, \infty]$.
The following assertions hold:
\begin{enumerate}
\item[(i)] The space $\TLA$ is complete with respect to the quasi-norm $\| \cdot \|_{\TLA}$,
           and continuously embedded into $\CC^{\Z \times \Z^d}$
           equipped with the topology of pointwise, i.e., componentwise, convergence.
 \item[(ii)] If $r := \min\{1, p, q\}$, then for any $\eps \in (0,1)$, the quasi-norm $\| \cdot \|_{\TLA}^\ast$
             introduced in \Cref{lem:borelsets} satisfies
             \[
               \bigl(\| a + b \|_{\TLA}^\ast\bigr)^r
               \leq \bigl(\| a \|_{\TLA}^\ast\bigr)^r + \bigl(\| b \|^\ast_{\TLA}\bigr)^r
             \]
             for all $a, b \in \TLA$.
\end{enumerate}
\end{lemma}

\section{Sufficient condition} \label{sec:sufficient}

The aim of this section is to prove the sufficient condition (\Cref{thm:sufficient_intro})
for the coincidence of discrete Triebel-Lizorkin spaces. Before doing so, we show the following
simple lemma that provides different equivalent formulations of this sufficient condition.

\begin{lemma} \label{lem:characterization}
  Let $A,B \in \GL(d, \R)$. Then the following are equivalent:
  \begin{enumerate}[label=(\roman*)]
    \item The set $\{B^j A^{-j} : j \in \Z \}$ is finite;
    \item The set $\{ A^j B^{-j} \,\,:\,\, j \in \Z \}$ is finite;
    \item There exists $k \in \N$ such that $A^k = B^k$.
  \end{enumerate}
\end{lemma}

\begin{proof}
The equivalence of (i) and (ii) is immediate since $(A^j B^{-j})^{-1} = B^j A^{-j}$,
so it remains to show the equivalence of (ii) and (iii).
For this, let $F := \{ A^j B^{-j} \,\,:\,\, j \in \Z \}$ and suppose that $F$ is finite.
Then the map $j \mapsto A^j B^{-j}$ cannot be injective from $\N$ to $F$, and
hence there exist $j,\ell \in \N$ with $j \neq \ell$ and $A^j B^{-j} = A^\ell B^{-\ell}$.
Without loss of generality, we may assume that $j > \ell$.
Then we also get that $A^j = A^{\ell} B^{j - \ell}$,
and thus $A^{j - \ell} = B^{j - \ell}$, so that setting $k := j - \ell \in \N$ shows (iii).

Conversely, suppose there exists $k \in \N$ with $A^k = B^k$.
Then induction shows that $A^{k \ell} = B^{k \ell}$ for all $\ell \in \N$,
and hence also $A^{-k \ell} = B^{-k \ell}$ for $\ell \in \N$,
which shows that $A^{k \ell} = B^{k \ell}$ for all $\ell \in \Z$.
Let $j \in \Z$ be arbitrary.
Then $j = \ell k + r$ for suitable $\ell \in \Z$ and $r \in \{ 0,\dots,k-1 \}$, and thus
\[
  A^j B^{-j}
  = A^r A^{\ell k} B^{-\ell k} B^{-r}
  = A^r B^{-r}
  .
\]
This shows that
\(
  \bigl\{ A^j B^{-j} \,\,:\,\, j \in \Z \bigr\}
  \subseteq \bigl\{ A^r B^{-r} \,\,:\,\, r \in \{ 0,\dots,k-1 \} \bigr\}
\)
is finite.
\end{proof}

The following theorem corresponds to \Cref{thm:sufficient_intro}.

\begin{theorem} \label{thm:sufficient}
If $A, B \in \GL(d, \R)$ are two expansive matrices such that $\{ A^j B^{-j} \,\,:\,\, j \in \Z \}$
is a finite set, then $\TLA = \TLB$ for all $\alpha \in \R$ and $p,q \in (0, \infty]$.
\end{theorem}

\begin{proof}
We will use that $\{ A^j B^{-j} \,\,:\,\, j \in \Z \}$ is finite
if and only if $\{ B^{j} A^{-j} : j \in \Z\}$ is finite; see \Cref{lem:characterization}.
Again by \Cref{lem:characterization}, there exists $k \in \N$ with $A^k = B^k$,
so that $(\det A)^k = (\det B)^k$ and hence $|\det A| = |\det B|$.
Let $N := \# \{ B^{j} A^{-j} : j \in \Z\}$ and write $\{ B^{j} A^{-j} : j \in \Z\} = \{ M_1, \dots, M_N \}$
for necessarily pairwise distinct invertible matrices $M_1, \dots, M_N$.
For $t \in \mathbb{N}$ with $t \leq N$, let $J_t := \{ j \in \Z : B^{j} A^{-j} = M_t \}$,
and note that $\Z = \bigcupdot_{1 \leq t \leq N} J_t$.

We split the proof into the cases $p < \infty$ and $p = \infty$.
\\~\\
\textbf{Case 1.} Let $p < \infty$. If $q < \infty$, then for arbitrary $c \in \CC^{\Z \times \Z^d}$,
\begin{align*}
  \| c \|_{\TLB}
  &= \bigg\|
       \bigg(
         \sum_{j \in \Z}
           |\det(B)|^{-jq(\alpha + 1/2)}
           \sum_{k \in \Z^d}
             |c_{j,k}|^q \Indicator_{B^j ([0,1]^d+k)}
       \bigg)^{\frac{1}{q}}
     \bigg\|_{L^p} \\
  &\asymp_{p,q, N} \sum_{t = 1}^N
                     \bigg\|
                       \bigg(
                         \sum_{j \in J_t}
                           |\det (A)|^{-jq (\alpha + 1/2)}
                           \sum_{k \in \Z^d}
                             |c_{j,k}|^q \Indicator_{M_t A^j ([0,1]^d + k)}
                       \bigg)^{\frac{1}{q}}
                     \bigg\|_{L^p}.
\end{align*}
Using that $\Indicator_{M_t A^j ([0,1]^d +k)} = \Indicator_{ A^j ([0,1]^d +k)} (M_t^{-1} \cdot)$,
a change of variable gives
\begin{align*}
  \| c \|_{\TLB}
  &\asymp_{p, q, N} \sum_{t = 1}^N
                      |\det (M_t)|^{1/p}
                      \bigg\|
                        \bigg(
                          \sum_{j \in J_t}
                            |\det (A)|^{-jq (\alpha + 1/2)}
                            \sum_{k \in \Z^d}
                              |c_{j,k}|^q \Indicator_{A^j ([0,1]^d + k)}
                        \bigg)^{\frac{1}{q}}
                      \bigg\|_{L^p} \\
  &\asymp_{p, A, B} \bigg\|
                      \bigg(
                        \sum_{j \in \Z}
                          |\det(A)|^{-jq(\alpha + 1/2)}
                          \sum_{k \in \Z^d}
                            |c_{j,k}|^q \Indicator_{A^j ([0,1]^d+k)}
                      \bigg)^{\frac{1}{q}}
                    \bigg\|_{L^p} \\
  &= \| c \|_{\TLA},
\end{align*}
where the second step used that $|\det (M_t)| \asymp 1$ for an implicit constant independent of $t$.

The case $q = \infty$ follows by similar arguments:
For arbitrary $c \in \CC^{\Z \times \Z^d}$ and $x \in \R^d$,,
\begin{align*}
  & \sup_{j \in \Z} \sup_{k \in \Z^d} |\det (B)|^{-j(\alpha + 1/2)} |c_{j,k}| \Indicator_{B^j([0,1]^d + k)} (x) \\
  &\quad \quad \asymp_{N}
               \sum_{t = 1}^N
                 \sup_{j \in J_t}
                   \sup_{k \in \Z^d}
                     |\det (B)|^{-j(\alpha + 1/2)} |c_{j,k}| \Indicator_{B^j([0,1]^d + k)} (x),
\end{align*}
so that a change of variable yields
\begin{align*}
  \| c \|_{\TLnaked^{\alpha}_{p, \infty}(B)}
  &= \bigg\|
       \sup_{j \in \Z}
         \sup_{k \in \Z^d}
           |\det (B)|^{-j(\alpha + 1/2)} |c_{j,k}| \Indicator_{B^j([0,1]^d + k)}
     \bigg\|_{L^p} \\
   &\asymp_{p, N} \sum_{t = 1}^N
                    \bigg\|
                      \sup_{j \in J_t}
                        \sup_{k \in \Z^d}
                          |\det (B)|^{-j(\alpha + 1/2)} |c_{j,k}| \Indicator_{A^j([0,1]^d + k)} (M^{-1}_t \cdot)
                    \bigg\|_{L^p} \\
   &\asymp_{p, A, B} \bigg\|
                       \sup_{j \in \Z}
                         \sup_{k \in \Z^d}
                           |\det (A)|^{-j(\alpha + 1/2)} |c_{j,k}| \Indicator_{A^j([0,1]^d + k)}
                     \bigg\|_{L^p} \\
   &= \| c \|_{\TLnaked^{\alpha}_{p,\infty} (A)},
\end{align*}
as required.
\\~\\
\textbf{Case 2.} Let $p = \infty$.
If $q < \infty$ and $0 < \varepsilon < 1$, then \Cref{lem:borelsets} yields, for $c \in \CC^{\Z \times \Z^d}$,
\begin{align*}
 & \| c \|_{\TLAi} \\
 & \asymp \sum_{t = 1}^N
            \inf
            \bigg\{
              \bigg\|
                \bigg(
                  \sum_{j \in J_t}
                    \sum_{k \in \Z^d}
                      \big( |\det(A)|^{-j(\alpha + 1/2)} |c_{j,k}| \Indicator_{E_{j,k}} \big)^q
                \bigg)^{\frac{1}{q}}
              \bigg \|_{L^{\infty}} : E_{j,k} \subseteq Q^A_{j,k} , \; \frac{|E_{j,k}|}{|Q^A_{j,k}|} > \varepsilon
            \bigg\}.
\end{align*}
Given a Borel set $E_{j,k} \subseteq Q_{j,k}^A$ with $|E_{j,k}|/|Q^A_{j,k}| > \varepsilon$,
define $E^*_{j,k} := A^{-j} E_{j,k} - k \subseteq [0, 1]^d$,
so that $E_{j,k} = A^j (E_{j,k}^* + k) = M_t^{-1} B^j (E_{j,k}^* + k) $ for $j \in J_t$,
and $|E^*_{j,k}| > \varepsilon$.
Then, using that
\[
  \Indicator_{E_{j,k}}
  = \Indicator_{M^{-1}_t  B^j (E^*_{j,k} + k)}
  = \Indicator_{ B^j (E^*_{j,k} + k)}(M_t \cdot)
  \quad \text{for all} \quad j \in J_t, k \in \Z^d,
\]
a change of variable yields
\begin{align*}
 & \bigg\|
     \bigg(
       \sum_{j \in J_t} \sum_{k \in \Z^d} \big( |\det(A)|^{-j(\alpha + 1/2)} |c_{j,k}| \Indicator_{E_{j,k}} \big)^q
     \bigg)^{\frac{1}{q}}
   \bigg \|_{L^{\infty}} \\
 & \quad \quad = \bigg\|
                   \bigg(
                     \sum_{j \in J_t}
                       \sum_{k \in \Z^d}
                         \big( |\det(A)|^{-j(\alpha + 1/2)} |c_{j,k}| \Indicator_{B^j(E^*_{j,k} + k)} \big)^q
                   \bigg)^{\frac{1}{q}}
                 \bigg \|_{L^{\infty}}.
\end{align*}
As $E_{j,k}$ runs through all subsets $E_{j,k} \subseteq Q^A_{j,k}$
with $|E_{j,k}|/|Q^A_{j,k}| > \varepsilon$, the set $E'_{j,k} := B^j (E^*_{j,k} + k)$
runs through all subsets $E'_{j,k} \subseteq Q_{j,k}^B$ with $|E'_{j,k}|/|Q^B_{j,k}| > \varepsilon$.
Thus, a combination of the above equivalences, together with $|\det(A)| = |\det (B)|$
and another application of \Cref{lem:borelsets}, yields
\begin{align*}
 &\| c \|_{\TLAi} \\
 &\asymp \sum_{t = 1}^N
           \inf
           \bigg\{
             \bigg\|
               \bigg(
                 \sum_{j \in J_t} \sum_{k \in \Z^d} \big( |\det(B)|^{-j(\alpha + 1/2)} |c_{j,k}| \Indicator_{F_{j,k}} \big)^q
               \bigg)^{\frac{1}{q}}
             \bigg \|_{L^{\infty}}
             :
             F_{j,k} \subseteq Q^B_{j,k} , \; \frac{|F_{j,k}|}{|Q^B_{j,k}|} > \varepsilon
           \bigg\} \\
 &\asymp \| c \|_{\TLnaked^{\alpha}_{\infty, q} (B)},
\end{align*}
which shows the claim for $q<\infty$.

The remaining case $q = \infty$ follows immediately from the fact that $|\det(A)| = |\det (B)|$.
\end{proof}

\section{Necessary condition}
\label{sec:necessary}

This section is devoted to proving  \Cref{thm:necessary_intro}.
We start by proving some simple lemmas that will be used in various parts of the proof.

\subsection{Two lemmas}

The first lemma provides two simple consequences of the coincidence of spaces.

\begin{lemma} \label{lem:basic_concidence}
Let $A, B \in \GL(d, \R)$ be expansive matrices, $\alpha_1, \alpha_2 \in \R$ and $p_1, p_2, q_1, q_2 \in (0, \infty]$.

If $\TLone(A) = \TLtwo(B)$, then the following assertions hold:
\begin{enumerate}
 \item[(i)] There exists $C \geq 1$ such that
            \[
              \frac{1}{C} \| c \|_{\TLone(A)} \leq \| c \|_{\TLtwo(B)} \leq C \| c \|_{\TLone(A)}
            \]
            for all $c \in \CC^{\Z \times \Z^d}$;
\item[(ii)] $|\det (A)|^{\alpha_1 + \frac{1}{2} - \frac{1}{p_1}} = |\det (B)|^{\alpha_2 + \frac{1}{2} - \frac{1}{p_2}}$.
\end{enumerate}
\end{lemma}

\begin{proof}
(i) The proof is analogous that of function spaces \cite[Lemma 5.2]{koppensteiner2023classification}.

If $\TLone(A) = \TLtwo(B)$, then the identity map $\iota : \TLone(A) \to \TLtwo(B)$
given by $c \mapsto c$ is well-defined, and it follows from the continuous embeddings
$\TLone(A), \TLtwo(B) \hookrightarrow \CC^{\Z \times \Z^d}$ (cf.\ \Cref{lem:basic})
that the graph of $\iota$ is closed.
Let $r := \min \{ 1, p_1, q_1, p_2, q_2 \}$.
Since $\TLone(A)$ and $\TLtwo(B)$ are complete with respect to their \enquote{natural} (quasi-)norms,
they are also complete, respectively, with respect to the equivalent $r$-norms
$\| \cdot \|_{\TLone(A)}^\ast$ and $\| \cdot \|_{\TLtwo(B)}^\ast$ introduced in \Cref{lem:borelsets};
cf.\ \Cref{lem:basic}.
Therefore, an application of the closed graph theorem (see, e.g., \cite[Theorem 2.15]{RudinFA})
implies that
\(
  \| c \|_{\TLtwo(B)}
  \asymp \| c \|_{\TLtwo(B)}^\ast
  \lesssim \| c \|_{\TLone(A)}^\ast \asymp \| c \|_{\TLone(A)}
\)
for all $c \in \CC^{\Z \times \Z^d}$.
The reverse inequality follows by symmetry.
\\~\\
(ii) Define $c := e_{j_0, 0}$ for $j_0 \in \Z$, where $(e_{j,k})_{j \in \mathbb{Z}, k \in \mathbb{Z}^d}$
denotes the standard basis for $\mathbb{C}^{\mathbb{Z} \times \mathbb{Z}^d}$.
If $p_1 < \infty$, then
\[
  \| c \|_{\TLone(A)}
  = \bigg\|
      |\det(A)|^{-{j_0} (\alpha_1 + \frac{1}{2})} \mathds{1}_{A^{j_0} [0, 1]^d}
    \bigg\|_{L^{p_1}}
  = |\det(A)|^{-j_0 (\alpha_1 + \frac{1}{2} - \frac{1}{p_1})},
\]
while if $p_1 = \infty$, then using that
\[
  \bigg\|
    |\det(A)|^{- j_0(\alpha_1 + \frac{1}{2})} \mathds{1}_{E_{j_0,0}}
  \bigg\|_{L^{\infty}}
  = |\det(A)|^{-j_0 (\alpha_1 + \frac{1}{2})}
  = |\det(A)|^{-j_0 (\alpha_1 + \frac{1}{2} - \frac{1}{p_1} )}
\]
for Borel sets $E_{j_0,0}$ of positive measure, it follows from an application of \Cref{lem:borelsets}
that $\| c \|_{\TLnaked^{\alpha_1}_{\infty, q_1}(A)} \asymp |\det(A)|^{-j_0 (\alpha_1 + \frac{1}{2} - \frac{1}{p_1})}$.
Similarly, $\| c \|_{\TLtwo(B)} \asymp |\det(B)|^{-j_0 (\alpha_2 + \frac{1}{2} - \frac{1}{p_2})}$.

Since $\| \cdot \|_{\TLone(A)} \asymp \| \cdot \|_{\TLtwo(B)}$ by assertion (i), it follows that
\[
  |\det(A)|^{-j_0 (\alpha_1 + \frac{1}{2} - \frac{1}{p_1})}
  \asymp |\det(B)|^{-j_0 (\alpha_2 + \frac{1}{2} - \frac{1}{p_2})}, \quad j_0 \in \Z,
\]
which easily implies the claim.
\end{proof}

\begin{lemma}\label{lem:pairwisedisjoint}
  Suppose $A, B \in \GL(d, \R)$ are expansive matrices such that $\{B^j A^{-j} : j \in \Z \}$ is infinite.
  Then, given any $N \in \N$, there exist $j_1, \dots, j_N \in \Z$ and $x_0 \in \R^d$ such that
  \[
   B^{j_t} A^{-j_t} x_0, \quad 1 \leq t \leq N,
  \]
  are pairwise distinct.
  Moreover, there exists $\varepsilon > 0$ such that
  \[
   B^{j_t} A^{-j_t} (x_0 + [-\eps, \eps]^d) , \quad 1 \leq t \leq N,
  \]
  are pairwise disjoint.
  In addition, for any $R>0$, the sets $B^{j_t} A^{-j_t} \mathcal{B}_{R\varepsilon} (R x_0)$,
  $1 \leq t \leq N$, are pairwise disjoint.
\end{lemma}

\begin{proof}
For proving the first claim, assume towards a contradiction that there do not exist
$j_1, \dots, j_N \in \Z$ and $x_0 \in \R^d$ such that the points $B^{j_t} A^{-j_t} x_0$, $ 1 \leq t \leq N$,
are pairwise distinct.
Then, for every $x_0 \in \R^d$, it follows that $\# \{ B^j A^{-j} x_0 : j \in \Z \} < N$.
In particular, this implies that the set $\{B^j A^{-j} e_i : j \in \Z \}$ is finite
for every standard basis vector $e_i$ with $1 \leq i \leq d$.
Setting $\mathcal{C}_i := \{B^j A^{-j} e_i : j \in \Z \}$, it follows that
$\# \{ B^j A^{-j} : j \in \Z \} \leq \Pi_{i = 1}^d \# \mathcal{C}_i < \infty$,
which is a contradiction.

For the remaining claims, let $j_1, \dots, j_N \in \Z$ and $x_0 \in \R^d$ be such that
$B^{j_t} A^{-j_t} x_0$, $ 1 \leq t \leq N$, are pairwise distinct.
Choose $R'>1$ such that $\max_{1 \leq t \leq N} \|B^{j_t} A^{-j_t} \| \leq R'$.
Moreover, choose some $\delta > 0$ satisfying
\[
  \delta < \frac{1}{2} \min_{t \neq t'} \| B^{j_t} A^{-j_t} x_0 - B^{j_{t'}} A^{-j_{t'}} x_0 \|,
\]
so that $\overline{\mathcal{B}}_{\delta} (B^{j_t} A^{-j_t} x_0)$, $1 \leq t \leq N$, are pairwise disjoint.
Then, choosing $0 < \varepsilon < \delta / (R' \sqrt{d})$ yields
\begin{align*}
  B^{j_t} A^{-j_t} (x_0 + [-\eps, \eps]^d)
  & \subseteq B^{j_t} A^{-j_t} \overline{\mathcal{B}}_{\varepsilon \sqrt{d}} (x_0)
    = B^{j_t} A^{-j_t} \overline{\mathcal{B}}_{ \varepsilon \sqrt{d}} (0) + B^{j_t} A^{-j_t} x_0 \\
  & \subseteq R' \overline{\mathcal{B}}_{ \varepsilon \sqrt{d}}(0) + B^{j_t} A^{-j_t} x_0
    \subseteq \mathcal{B}_{\delta} (B^{j_t} A^{-j_t} x_0).
\end{align*}
Lastly, note that
\[
  B^{j_t} A^{-j_t} \mathcal{B}_{R \varepsilon} (R x_0)
  = R \cdot (B^{j_t} A^{-j_t} \mathcal{B}_{\varepsilon} (x_0))
  \subseteq R \cdot \bigl(B^{j_t} A^{-j_t} (x_0 + [-\eps, \eps]^d)\bigr),
\]
which proves the final claim.
\end{proof}

\subsection{Key results}

In this section, we prove the  various necessary conditions for the coincidence
of discrete Triebel-Lizorkin spaces associated to possibly different exponents and dilation matrices.
For clarity, we prove these necessary conditions by establishing various subresults.

We start by showing that $p_1 = p_2$ whenever $\TLone(A) = \TLtwo(B)$.

\begin{proposition} \label{prop:p1=p2}
Let $A, B \in \GL(d, \R)$ be expansive, $\alpha_1, \alpha_2 \in \R$ and $p_1, p_2, q_1, q_2 \in (0, \infty]$.

If $\TLone(A) = \TLtwo(B)$, then $p:= p_1 = p_2$
and $|\det(A)|^{\alpha_1 + \frac{1}{2} - \frac{1}{p}} = |\det(B)|^{\alpha_2 + \frac{1}{2} - \frac{1}{p}}$.
\end{proposition}

\begin{proof}
Let $(a_k)_{k \in \Z^d} \in \CC^{\Z^d}$ be arbitrary and define $c \in \CC^{\Z \times \Z^d}$
by $c_{j,k} = \delta_{0, j} a_k$.
We will show that $\| c \|_{\TLone(A)} \asymp \| a \|_{\ell^{p_1}}$.
Once this is shown, it follows by symmetry and an application of \Cref{lem:basic_concidence}(i)
that $\| a \|_{\ell^{p_1}} \asymp \| c \|_{\TLone(A)} \asymp \| c \|_{\TLtwo(B)} \asymp \| a \|_{\ell^{p_2}}$,
hence $p := p_1 = p_2$, and $|\det(A)|^{\alpha_1 + \frac{1}{2} - \frac{1}{p}} = |\det(B)|^{\alpha_2 + \frac{1}{2} - \frac{1}{p}}$
by \Cref{lem:basic_concidence}(ii).

For showing that $\| c \|_{\TLone(A)} \asymp \| a \|_{\ell^{p_1}}$,
we will consider the cases $p_1 < \infty$ and $p_1 = \infty$.
\\~\\
\textbf{Case 1.} Let $p_1 \in (0, \infty)$.
If $q_1 < \infty$, then
\begin{align*}
  \| c \|_{\TLone(A)}
  &= \bigg\|
       \bigg(
         \sum_{j \in \Z} |\det(A)|^{-j q_1(\alpha_1 + \frac{1}{2})} \sum_{k \in \Z^d} |c_{j,k}|^{q_1} \Indicator_{Q_{j,k}^A}
       \bigg)^{1/{q_1}}
      \bigg\|_{L^{p_1}} \\
  &= \bigg\|
       \bigg(
         \sum_{k \in \Z^d} |a_k|^{q_1} \Indicator_{[0,1]^d + k}
       \bigg)^{1/{q_1}}
     \bigg\|_{L^{p_1}} \\
  &= \bigg\| \sum_{k \in \Z^d} |a_k |^{p_1} \Indicator_{[0,1]^d + k} \bigg\|^{1/{p_1}}_{L^1}
  = \bigg( \sum_{k \in \Z^d} |a_k|^{p_1} \bigg)^{1/{p_1}}.
\end{align*}
Similarly, if $q_1 = \infty$, then
\begin{align*}
  \| c \|_{\TLnaked^{\alpha_1}_{p_1, \infty}(A)}
  &= \bigg\|
       \sup_{j \in \Z}
         \sup_{k \in \Z^d}
           |\det(A)|^{-j (\alpha_1 + \frac{1}{2})} |c_{j,k}| \Indicator_{Q^A_{j,k}}
     \bigg\|_{L^{p_1}} \\
  &= \bigg\|
       \sup_{k \in \Z^d}
         |a_k| \Indicator_{[0,1]^d + k}
     \bigg\|_{L^{p_1}}
   =\bigg\| \sum_{k \in \Z^d} |a_k| \Indicator_{[0,1]^d + k} \bigg\|_{L^{p_1}} \\
  &= \bigg( \sum_{k \in \Z^d} |a_k|^{p_1} \bigg)^{1/{p_1}},
\end{align*}
where the penultimate step used that the sets $[0,1]^d + k$, $k \in \Z^d$
are pairwise disjoint up to null-sets.
\\~\\
\textbf{Case 2.} Let $p_1 = \infty$ and $0 < \varepsilon < 1$.
If $q_1 < \infty$, then it follows by \Cref{lem:borelsets} that
\[
  \| c \|_{\TLnaked_{\infty, q_1}^{\alpha_1}(A)}
  \asymp \inf
         \bigg\{
           \bigg\|
             \bigg(
               \sum_{k \in \Z^d} \big(  |a_k| \Indicator_{E_{0,k}} \big)^{q_1}
             \bigg)^{\frac{1}{q_1}}
           \bigg \|_{L^{\infty}}
           :
           E_{0,k} \subseteq Q^A_{0,k} , \; \frac{|E_{0,k}|}{|Q^A_{0,k}|} > \varepsilon
         \bigg\}
\]
where $E_{0,k} \subseteq Q^A_{0,k} = [0, 1]^d + k$ are Borel sets.
Since the sets $E_{0,k}$, $k \in \Z^d$, are pairwise disjoint up to null-sets, a direct calculation gives
\begin{align*}
  \bigg\|
    \bigg(
      \sum_{k \in \Z^d} \big(  |a_k| \Indicator_{E_{0,k}} \big)^{q_1}
    \bigg)^{\frac{1}{q_1}}
  \bigg \|_{L^{\infty}}
  = \| a \|_{\ell^{\infty}},
\end{align*}
which shows that $\| c \|_{\TLnaked_{\infty, q_1}^{\alpha_1}(A)} \asymp \| a \|_{\ell^{\infty}}$ whenever $q_1 < \infty$.
The remaining case $p_1 = q_1 = \infty$ is immediate.
\end{proof}

We next show that necessarily $p = q_1 = q_2$ whenever $\TLonep(A) = \TLtwop(B)$
and the set $\{ B^{j} A^{-j} : j \in \Z\}$ is infinite.
This is the most difficult part of the proof of \Cref{thm:necessary_intro}.

\begin{theorem} \label{thm:p=q}
Let $A, B \in \GL(d, \R)$ be expansive, $\alpha_1, \alpha_2 \in \R$, $p \in (0, \infty]$ and $ q_1, q_2 \in (0, \infty]$.

If $\{B^j A^{-j} : j \in \Z\}$ is infinite and $\TLonep(A) = \TLtwop(B)$, then $p = q_1 = q_2$.
\end{theorem}

Our proof for the case $p < \infty$ of \Cref{thm:p=q} is based on some ideas
used for the construction of sequences in the proof of \Cref{thm:triebel}
(see \cite[Proposition 5.26]{triebel2006theory}).

\begin{proof}[Proof of \Cref{thm:p=q}]
 Since $ \{B^j A^{-j} : j \in \Z\}$ is infinite, \Cref{lem:pairwisedisjoint} shows for
 any given $N \in \mathbb{N}$ that there exist $\varepsilon > 0$, $j_1, \dots, j_N \in \Z$
 and $x_0 \in \R^d$ such that, for any $R>0$, the sets $B^{j_t} A^{-j_t} \mathcal{B}_{R \varepsilon} (R x_0)$,
 $1 \leq t \leq N$, are pairwise disjoint.
 In particular, $j_t \neq j_{t'}$ for $t \neq t'$.
 We will choose the value of $R > 0$ depending on the cases $p < \infty$ and $p = \infty$, which we treat separately.
\\~\\
 \textbf{Case 1.} We first consider the case $p < \infty$. In this case,
 we let $R' := \max_{1 \leq t \leq N} \sqrt{d} \| A^{j_t} \|$, and fix some $R \geq \frac{2}{\varepsilon} R'$.
Define $P_R :=  \mathcal{B}_{\frac{R \varepsilon}{2}} (R x_0)$ and set
\[
 I_{t, R} := \big\{ k \in \Z^d : A^{j_t} ([0,1]^d + k) \cap P_R \neq \emptyset \big\}, \quad 1 \leq t \leq N.
\]
Let $(\tau_t)_{t = 1}^N \in \mathbb{R}^N$ be arbitrary and define $c \in \CC^{\Z \times \Z^d}$ by
\begin{align*}
 c_{j,k} :=
 \begin{cases}
  |\tau_t| \cdot  |\det(A)|^{j_t (\alpha_1 + \frac{1}{2} - \frac{1}{p})},
  & \text{if $j = j_t$ for a (unique) $1 \leq t \leq N$ and $k \in I_{t, R}$} \\
  0,
  & \text{otherwise.}
 \end{cases}
\end{align*}
We will show that $p = q_1 = q_2$ by comparing the norm of $c$ for $\TLonep(A)$ and $\TLtwop(B)$.

We start by estimating the norm of $c$ for the space $\TLonep(A)$.
For this, consider the set $\Omega_{t, R} := \bigcup_{k \in I_{t, R}} A^{j_t} ([0,1]^d + k)$
and note that $P_R \subseteq \Omega_{t,R}$.
Second, note that
\[
  \diam( A^{j_t} ([0, 1]^d + k) ) \leq \| A^{j_t} \| \sqrt{d} \leq R'
\]
and that if two sets $\Omega, \Omega' \subseteq \mathbb{R}^d$ satisfy $\Omega \cap \Omega' \neq \emptyset$
and $\rho = \diam(\Omega)$, then $\Omega \subseteq \overline{\mathcal{B}}_{\rho} (\Omega')$.
Therefore, if $k \in I_{t,R}$, then
\[
  A^{j_t} ([0,1]^d + k)
  \subseteq \overline{\mathcal{B}}_{R'} (P_R)
  \subseteq \mathcal{B}_{R\frac{\varepsilon}{2} + R'} (R x_0)
  \subseteq \mathcal{B}_{R\varepsilon} (R x_0) ,
\]
where the last inclusion uses that $R' \leq R\frac{\varepsilon}{2}$.
In combination, this shows  $P_R \subseteq \Omega_{t,R} \subseteq \mathcal{B}_{R \varepsilon} (R x_0)$,
whence $|\Omega_{t,R}| \asymp_d (R\varepsilon)^d$.
On the one hand, if $q_1< \infty$, then a direct calculation gives
\begin{align*}
  \| c \|_{\TLonep(A)}
  &= \bigg\|
       \bigg(
         \sum_{j \in \Z}
           |\det(A)|^{-j q_1(\alpha_1 + \frac{1}{2})}
           \sum_{k \in \Z^d}
             |c_{j,k}|^{q_1} \Indicator_{Q_{j,k}^A}
       \bigg)^{1/{q_1}}
     \bigg\|_{L^{p}} \\
  &= \bigg\|
       \bigg(
         \sum_{t = 1}^N
           |\det(A)|^{-j_t \frac{q_1}{p}}
           |\tau_t|^{q_1}
           \sum_{k \in I_{t,R}}
             \Indicator_{A^{j_t} ([0,1]^d + k)}
       \bigg)^{1/q_1}
     \bigg\|_{L^{p}} \\
  &= \bigg\|
       \bigg(
         \sum_{t = 1}^N |\det(A)|^{-j_t \frac{q_1}{p}} |\tau_t|^{q_1}  \Indicator_{\Omega_{t,R}}
       \bigg)^{1/q_1}
     \bigg\|_{L^{p}} \\
  &\geq \| \Indicator_{P_R} \|_{L^{p}} \big\| \big(|\det (A)|^{-j_t / p} \tau_t \big)_{t = 1}^N \big\|_{\ell^{q_1}} \\
  &\gtrsim_{d,p} (R\varepsilon)^{d/p} \big\| \big(|\det (A)|^{-j_t / p} \tau_t \big)_{t = 1}^N \big\|_{\ell^{q_1}},
\end{align*}
and, similarly,
\begin{align*}
   \| c \|_{\TLonep(A)}
   &= \bigg\|
        \bigg(
          \sum_{t = 1}^N |\det(A)|^{-j_t \frac{q_1}{p}} |\tau_t|^{q_1}  \Indicator_{\Omega_{t,R}}
        \bigg)^{1/q_1}
      \bigg\|_{L^{p}} \\
  &\leq \| \Indicator_{\mathcal{B}_{R\varepsilon} (R x_0)} \|_{L^{p}}
        \big\| \big(|\det (A)|^{-j_t / p} \tau_t \big)_{t = 1}^N \big\|_{\ell^{q_1}} \\
  &\lesssim_{d,p} (R\varepsilon)^{d/p}
                  \big\| \big(|\det (A)|^{-j_t / p} \tau_t \big)_{t = 1}^N \big\|_{\ell^{q_1}}.
\end{align*}
On the other hand, if $q_1 = \infty$, then using that
\(
  \sup_{k \in I_{t,R}} \Indicator_{A^{j_t} ([0,1]^d + k)} = \Indicator_{\Omega_{t,R}}
\)
almost everywhere for $1 \leq t \leq N$ and that
$\Indicator_{P_R} \leq \Indicator_{\Omega_{t,R}} \leq \Indicator_{\mathcal{B}_{R \varepsilon} (R x_0)}$,
it follows that
\begin{align*}
   \| c \|_{\TLonep(A)}
   &= \bigg\|
        \sup_{1 \leq t \leq N}
          |\det(A)|^{-\frac{j_t}{p}}
          |\tau_t|
          \sup_{k \in I_{t,R}}
            \Indicator_{A^{j_t} ([0,1]^d + k)}
      \bigg\|_{L^{p}} \\
   &=  \bigg\|
         \sup_{1 \leq t \leq N} |\det(A)|^{-\frac{j_t}{p}} |\tau_t| \Indicator_{\Omega_{t,R}}
       \bigg\|_{L^{p}} \\
   &\asymp_{d,p} (R\varepsilon)^{d/p}
                 \big\| \big(|\det (A)|^{-j_t / p} \tau_t \big)_{t = 1}^N \big\|_{\ell^{q_1}}.
\end{align*}
Thus,
\(
  \| c \|_{\TLonep(A)}
  \asymp_{d,p} (R\varepsilon)^{d/p}
               \big\| \big(|\det (A)|^{-j_t / p} \tau_t \big)_{t = 1}^N \big\|_{\ell^{q_1}}
\)
for any possible $q_1 \in (0, \infty]$.

For estimating the norm of $c$ for the space $\TLtwop(B)$,
define $\Lambda_{t,R} := \bigcupdot_{k \in I_{t,R}} B^{j_t} ([0,1]^d +  k)$.
Note that $\Lambda_{t,R} = B^{j_t} A^{-j_t} \Omega_{t, R}$ and that the sets $\Lambda_{t,R}$, $1 \leq t \leq N$,
are pairwise disjoint as $\Omega_{t,R} \subseteq \mathcal{B}_{R \varepsilon} (R x_0)$
and $B^{j_t} A^{-j_t}\mathcal{B}_{R \varepsilon} (R x_0)$ are pairwise disjoint for $1 \leq t \leq N$.
Using that $|\det(A)|^{\alpha_1 + \frac{1}{2} - \frac{1}{p}} = |\det(B)|^{\alpha_2 + \frac{1}{2} - \frac{1}{p}}$ (cf.\ \Cref{prop:p1=p2}),
a direct calculation yields for the case $q_2 < \infty$ that
\begin{align*}
   \| c \|_{\TLtwop(B)}
   &= \bigg\|
        \bigg(
          \sum_{t= 1}^N
            |\det(B)|^{-j_t q_2(\alpha_2 + \frac{1}{2} - \frac{1}{p})}
            |\det(B)|^{- j_t \frac{q_2}{p}}
            \sum_{k \in \Z^d}
              |c_{j_t,k}|^{q_2} \Indicator_{Q_{j_t,k}^B}
        \bigg)^{1/{q_2}}
      \bigg\|_{L^{p}} \\
  &= \bigg\|
       \bigg(
         \sum_{t = 1}^N
         |\det(B)|^{-j_t \frac{q_2}{p}}
         |\tau_t|^{q_2}
         \sum_{k \in I_{t,R}}
           \Indicator_{B^{j_t} ([0,1]^d + k)}
       \bigg)^{1/q_2}
     \bigg\|_{L^{p}} \\
  &= \bigg\|  \sum_{t = 1}^N |\det(B)|^{-j_t } |\tau_t|^{p} \Indicator_{\Lambda_{t,R}}  \bigg\|_{L^{1}}^{1/p}
   = \bigg( \sum_{t = 1}^N |\det(B)|^{-j_t } |\tau_t|^{p} |\Lambda_{t,R}|  \bigg)^{1/p} \\
  &\asymp_{d,p} (R\varepsilon)^{d/p} \big\| \big(|\det(A)|^{-\frac{j_t}{p}} \tau_t \big)_{t = 1}^N \big\|_{\ell^{p}},
\end{align*}
where the last step used that $|\Lambda_{t,R}| \asymp_d \big( |\det(B)| / |\det(A)| \big)^{j_t} (R\varepsilon)^d$
since $\Lambda_{t,R} = B^{j_t} A^{-j_t} \Omega_{t,R}$.
The estimate
\[
  \| c \|_{\TLnaked^{\alpha_2}_{p, \infty} (B)}
  \asymp (R\varepsilon)^{d/p} \big\| \big(|\det(A)|^{-\frac{j_t}{p}} \tau_t \big)_{t = 1}^N \big\|_{\ell^{p}}
\]
for the case $q_2 = \infty$ is shown using similar arguments.

A combination of the above obtained estimates with \Cref{lem:basic_concidence}(i) thus yields
\begin{align*}
  (R\varepsilon)^{d/p} \big\| \big(|\det(A)|^{-\frac{j_t}{p}} \tau_t \big)_{t = 1}^N \big\|_{\ell^{p}}
  &\asymp \| c \|_{\TLtwop(B)} \\
  &\asymp \| c \|_{\TLonep(A)}
   \asymp (R\varepsilon)^{d/p} \big\| \big(|\det (A)|^{-\frac{j_t}{p}} \tau_t \big)_{t = 1}^N \big\|_{\ell^{q_1}},
\end{align*}
which implies that $q_1 = p$ since $N \in \N$ and $\tau = (\tau_t)_{t=1}^N \in \mathbb{R}^N$
were chosen arbitrary and the implied constants do not depend on $N, R, \varepsilon$ or $\tau$.
Since the condition that $\{B^j A^{-j} : j \in \Z\}$ is finite is symmetric in $A, B$,
it follows by symmetry that also $q_2 = p$.
\\~\\
\textbf{Case 2.} Suppose that $p = \infty$.
Throughout, we fix some $\delta \in (0, 1/6)$ and choose $\ell_0 \in \N$ such that
$A^{-\ell} [0, 1]^d \subseteq [-\delta, \delta]^d$ for all $\ell \geq \ell_0$,
which is possible since $A$ is expansive and hence $\| A^{-j} \| \to 0$ as $j \to \infty$;
this follows from the spectral radius formula, since the spectral radius of $A^{-1}$ satisfies $\rho(A^{-1}) < 1$.

Now, given $N \in \N$, choose $\eps > 0$ and $j_1,\dots,j_N \in \Z$ and $x_0 \in \R^d$
such that the sets $B^{j_t} A^{-j_t} \mathcal{B}_{R \varepsilon} (R x_0)$, $1 \leq t \leq N$,
are pairwise disjoint for all $R > 0$; this is possible by \Cref{lem:pairwisedisjoint}.
We then choose $j_0 \in \Z$ such that $j_0 \geq \ell_0 + \max_{1 \leq t \leq N} j_t$.
With this choice, we set $R := 10 \sqrt{d} \| A^{j_0} \| / \eps$ and $R_0 := R/10$.

Choose next $k_0 \in \Z^d$ such that $R x_0 \in A^{j_0} ([0,1]^d + k_0)$.
Then, since
\[
  \diam(A^{j_0} [0, 1]^d + k_0) \leq \| A^{j_0}\| \sqrt{d} \leq R_0 \eps
  ,
\]
it follows that
$A^{j_0} ([0,1]^d + k_0) \subseteq \overline{\mathcal{B}}_{R_0 \varepsilon} (R x_0) \subseteq \mathcal{B}_{R \varepsilon} (R x_0)$,
and hence it follows that also the sets $B^{j_t} A^{-j_t} A^{j_0} ([0,1]^d + k_0)$ are pairwise disjoint.
Finally, we set $P_{\delta} := A^{j_0} ([\frac{1}{2} - \delta, \frac{1}{2} + \delta)^d + k_0)$ and define
\[
  I_{t, \delta}
  := \{ k \in \Z^d : A^{j_t} ([0, 1]^d + k) \cap P_{\delta} \neq \emptyset \}, \quad 1 \leq t \leq N.
\]
Similar to Case 1, we define a sequence $c \in \CC^{\Z \times \Z^d}$ by
\begin{align*}
 c_{j,k} :=
 \begin{cases}
  |\det(A)|^{j_t (\alpha_1 + \frac{1}{2})} \cdot |\tau_t|,
  & \text{if $j = j_t$ for a (unique) $1 \leq t \leq N$ and $k \in I_{t, \delta}$} \\
  0,
  & \text{otherwise},
 \end{cases}
 \end{align*}
 where $\tau = (\tau_t)_{t=1}^N$ is an arbitrary given sequence in $\mathbb{R}^N$.

 For showing that $q_1 = q_2 = p$, we will estimate the norms $\| c \|_{\TLonep(A)}$ and $\| c \|_{\TLtwop(B)}$.
 For estimating $\| c \|_{\TLonep(A)}$, we define $\Omega_{t, \delta} := \bigcup_{k \in I_{t,\delta}} A^{j_t} ([0,1]^d + k)$.
 Then clearly $P_{\delta} \subseteq \Omega_{t, \delta}$,
 and we claim that $\Omega_{t, \delta} \subseteq Q_{j_0, k_0}^A = A^{j_0} ([0,1]^d + k_0)$.
 Indeed, note that
 if $k \in I_{t, \delta}$, then
 \begin{align*} k &\in A^{j_0 - j_t} k_0 + A^{j_0 - j_t} \big( [1/2 - \delta, 1/2 + \delta)^d - A^{j_t - j_0} [0, 1]^d \big) \\
  &\subseteq A^{j_0 - j_t} k_0 + A^{j_0 - j_t} \big( [1/2- \delta, 1/2 + \delta)^d - [- \delta, \delta]^d \big) \\
  &\subseteq A^{j_0 - j_t} k_0 + A^{j_0 - j_t} [1/2 - 2 \delta, 1/2 + 2 \delta)^d,
 \end{align*}
where we used that $j_t - j_0 \leq - \ell_0$ and $A^{-\ell} [0,1]^d \subseteq [-\delta, \delta]^d$ for all $\ell \geq \ell_0$.
Using again that $A^{j_t - j_0} [0, 1]^d \subseteq [ - \delta, \delta]^d$ and that $\delta < 1/6$, we finally see that
  \begin{align*}
  k + [0, 1]^d &\subseteq A^{j_0 - j_t} k_0 + A^{j_0 - j_t} \big( [1/2- 2\delta, 1/2 + 2\delta)^d + A^{j_t - j_0} [0, 1]^d \big) \\
  &\subseteq A^{j_0 - j_t} k_0 + A^{j_0 - j_t} [1/2 - 3 \delta, 1/2 + 3 \delta]^d \\
  &\subseteq A^{j_0 - j_t} k_0 + A^{j_0 - j_t} [0, 1]^d,
\end{align*}
whence $A^{j_t} \big([0,1]^d + k \big) \subseteq A^{j_0} \big ([0,1]^d + k_0 \big) = Q_{j_0, k_0}^{A}$
for any $k \in I_{t,\delta}$, as claimed.
As $j_0 \geq \ell_0 + \max_{1 \leq t \leq N} j_t$, we have $j_0 \geq j_t$ for each $1 \leq t \leq N$.
Hence, using the definition of $\TLnaked_{\infty, q_1}^{\alpha_1} (A)$ for $q_1 < \infty$, we estimate
\begin{align*}
  \| c \|_{\TLnaked_{\infty, q_1}^{\alpha_1} (A)}
  &\geq \bigg(
          \dashint_{Q^A_{j_0, k_0}}
            \sum_{j \in \Z, j \leq j_0}
              |\det(A)|^{- j q_1(\alpha_1+\frac{1}{2})}
              \sum_{k \in \Z^d}
                |c_{j,k}|^{q_1} \Indicator_{Q_{j,k}^A} (x)
          \; dx
        \bigg)^{\frac{1}{q_1}} \\
  &= \bigg(
       \dashint_{Q^A_{j_0, k_0}}
         \sum_{t = 1}^N
           |\tau_t|^{q_1}
           \sum_{k \in I_{t, \delta}}
           \Indicator_{Q_{j_t,k}^A} (x)
       \; dx
     \bigg)^{\frac{1}{q_1}} \\
  &= \bigg(
       \dashint_{Q^A_{j_0, k_0}} \sum_{t = 1}^N |\tau_t|^{q_1} \Indicator_{\Omega_{t, \delta}} (x) \; dx
     \bigg)^{\frac{1}{q_1}}
     \geq \bigg( \frac{|P_{\delta}|}{|Q_{j_0, k_0}^A|} \sum_{t = 1}^N |\tau_t|^{q_1} \bigg)^{1/q_1} \\
  &\gtrsim_{\delta, q_1} \| \tau \|_{\ell^{q_1}}.
\end{align*}
Clearly, $\| c \|_{\TLnaked^{\alpha_1}_{\infty, \infty}(A)} = \| \tau \|_{\ell^{\infty}}$,
so that $\| c \|_{\TLnaked_{\infty, q_1}^{\alpha_1}(A)} \gtrsim \| \tau \|_{\ell^{q_1}}$
for arbitrary $q_1 \in (0, \infty]$.
Here, we crucially used that $\frac{|P_{\delta}|}{|Q_{j_0, k_0}^A|} = (2 \delta)^d$
is independent of the choice of $N \in \N$ and of $\tau \in \R^N$.

We next provide an upper bound for $\| c \|_{\TLnaked^{\alpha_2}_{\infty, q_2} (B)}$.
Let $\Lambda_{t, \delta} := \bigcup_{k \in I_{t, \delta}} B^{j_t} ([ 0, 1]^d + k)$, and observe that
$\Lambda_{t, \delta} = B^{j_t} A^{-j_t} \Omega_{t, \delta} \subseteq B^{j_t} A^{- j_t} Q_{j_0, k_0}^A$,
where  the inclusion $\Omega_{t, \delta} \subseteq Q_{j_0, k_0}^A $ was shown already above.
In particular, since we showed towards the beginning of Part~2 of the proof that the
sets $B^{j_t} A^{-j_t} A^{j_0} ([0,1]^d + k_0)$, $1 \leq t \leq N$, are pairwise disjoint,
this implies that the sets $\Lambda_{t, \delta}$, $1 \leq t \leq N$, are pairwise disjoint as well.
Using this, together with the fact that $|\det(A)|^{\alpha_1 + \frac{1}{2}} = |\det(B)|^{\alpha_2 + \frac{1}{2}}$ (cf.\ \Cref{prop:p1=p2}),
a direct calculation entails for $q_2<\infty$ that
\begin{align*}
  \| c \|_{\TLnaked^{\alpha_2}_{\infty, q_2} (B)}
  &= \sup_{Q \in \mathcal{Q}^B}
     \bigg(
       \dashint_Q
       \bigg[
         \bigg(
           \sum_{\substack{j \in \Z \\ j \leq \scale_B(Q)}}
             |\det(B)|^{-j q_2 (\alpha_2 + \frac{1}{2})}
             \sum_{k \in \Z^d}
               |c_{j,k}|^{q_2} \Indicator_{Q^B_{j,k}} (x)
         \bigg)^{1/{q_2}}
       \bigg]^{q_2}
       \; dx
     \bigg)^{1/{q_2}} \\
  &\leq \bigg\|
          \bigg(
            \sum_{j \in \Z}
              |\det(B)|^{-j q_2 (\alpha_2 + \frac{1}{2})}
              \sum_{k \in \Z^d}
                |c_{j,k}|^{q_2} \Indicator_{Q^B_{j,k}}
          \bigg)^{\frac{1}{q_2}}
        \bigg\|_{L^{\infty}} \\
  &= \bigg\|
       \bigg(
         \sum_{t = 1}^N |\tau_t|^{q_2} \sum_{k \in I_{t,\delta}}  \Indicator_{Q^B_{j_t,k}}
       \bigg)^{\frac{1}{q_2}}
     \bigg\|_{L^{\infty}}
   = \bigg\|
       \bigg(
         \sum_{t = 1}^N |\tau_t|^{q_2} \Indicator_{\Lambda_{t, \delta}}
       \bigg)^{\frac{1}{q_2}}
     \bigg\|_{L^{\infty}} \\
  &= \| \tau \|_{\ell^{\infty}}.
\end{align*}
Clearly, also $\| c \|_{\TLnaked^{\alpha_2}_{\infty, \infty}(B)} = \| \tau \|_{\ell^{\infty}}$.

A combination of the estimates obtained above with \Cref{lem:basic_concidence}(i) gives
\[
  \| \tau \|_{\ell^{\infty}}
  \leq \| \tau \|_{\ell^{q_1}}
  \lesssim \| c \|_{\TLnaked_{\infty, q_1}^{\alpha_1}(A)}
  \asymp \| c \|_{\TLnaked_{\infty, q_2}^{\alpha_2} (B)}
  \leq \| \tau \|_{\ell^{\infty}},
\]
which implies that $\| \tau \|_{\ell^{q_1}} \asymp \| \tau \|_{\ell^{\infty}}$
for arbitrary $N \in \N$ and $\tau \in \R^N$ with an implicit constant independent of $N, \tau$.
Thus, $q_1 = \infty$.
Since the condition of $\{ B^j A^{-j} : j \in \Z \}$ being infinite is symmetric in $A, B$,
it follows by symmetry that also $q_2 = \infty$, so that $p = q_1 = q_2$.
This completes the proof.
\end{proof}

Lastly, we treat the case when $\{B^j A^{-j} : j \in \Z\}$ is finite.

\begin{proposition} \label{prop:q1=q2}
Let $A, B \in \GL(d, \R)$ be expansive matrices, $\alpha_1, \alpha_2 \in \R$, $p \in (0, \infty]$ and $ q_1, q_2 \in (0, \infty]$.

If $\{B^j A^{-j} : j \in \Z\}$ is finite and $\TLonep(A) = \TLtwop(B)$, then $\alpha_1 = \alpha_2$ and $q_1 = q_2$.
\end{proposition}
\begin{proof}
If $\{B^j A^{-j} : j \in \Z\}$ is finite, then necessarily $|\det(A)| = |\det(B)|$
(cf.\ the proof of \Cref{thm:sufficient}), and hence it follows by \Cref{lem:basic_concidence}
that  $|\det(A)|^{\alpha_1 + \frac{1}{2} - \frac{1}{p}} = |\det(A)|^{\alpha_2 + \frac{1}{2} - \frac{1}{p}}$.
Therefore, since $|\det(A)| > 1$, we get $\alpha := \alpha_1 = \alpha_2$.
Moreover, an application of \Cref{thm:sufficient} (which is applicable by \Cref{lem:characterization})
yields that $\TLnaked_{p, q_1}^{\alpha} (A) = \TLnaked_{p, q_2}^{\alpha}(B) = \TLnaked_{p, q_2}^{\alpha}(A)$.

 Given an arbitrary sequence $\tau = (\tau_t)_{t \in \mathbb{N}_0} \in \CC^{\N_0}$,
 define $c \in \CC^{\Z \times \Z^d}$ by
\begin{align*}
 c_{j,k} :=
 \begin{cases}
  |\det(A)|^{j(\alpha + \frac{1}{2})} |\tau_{-j}|, & \text{if $j \leq 0$ and $k \in I_j$} \\
  0, & \text{otherwise},
 \end{cases}
\end{align*}
where $I_j := \{ k \in \Z^d : Q^A_{j,k} \cap [0,1]^d \neq \emptyset \}$ for $j \in \Z$.
We let $\Omega_j := \bigcup_{k \in I_j} Q_{j,k}^A$ for $j \leq 0$.
Clearly, $[0, 1]^d \subseteq \Omega_j$.
On the other hand, if $j \leq 0$ and $k \in I_j$,
then $Q^A_{j,k} = A^j([0,1]^d + k) \subseteq \overline{\mathcal{B}}_R ([0,1]^d)$
for $R := \sqrt{d} \sup_{j \leq 0} \| A^j \|$, which is finite since $A$ is expansive,
and thus $Q_{j,k}^A \subseteq \mathcal{B}_{R'} (0)$ for some $R' > 0$ (only depending on $d, A$).
In combination, this yields $[0,1]^d \subseteq \Omega_j \subseteq \mathcal{B}_{R'}(0)$.

We now show that $\| c \|_{\TLA} \asymp \| \tau \|_{\ell^q}$ for any $q \in (0, \infty]$,
with implicit constant independent of $\tau$.
For this, we distinguish the cases $p < \infty$ and $p = \infty$.
\\~\\
\textbf{Case 1.} Suppose $p < \infty$.
For $q \in (0, \infty)$, the fact that $\Indicator_{[0,1]^d} \leq \Indicator_{\Omega_j} \leq \Indicator_{\mathcal{B}_{R'} (0)}$
for all $j \leq 0$ yields
\begin{align*}
 \| c \|_{\TL(A)}
 &= \bigg\| \bigg( \sum_{j \leq 0} |\tau_{-j}|^{q} \sum_{k \in I_j} \Indicator_{Q_{j,k}^A} \bigg)^{1/q} \bigg\|_{L^{p}} \\
 &= \bigg\| \bigg( \sum_{j \leq 0} |\tau_{-j}|^{q} \Indicator_{\Omega_j} \bigg)^{1/q} \bigg\|_{L^{p}} \\
 &\leq \bigg( \sum_{j \in \mathbb{N}_0} |\tau_{j}|^{q} \bigg)^{1/q} \| \Indicator_{\mathcal{B}_{R'} (0)} \|_{L^{p}}
\end{align*}
and $\| c \|_{\TL(A)} \geq \| \tau \|_{\ell^{q}} \| \Indicator_{[0,1]^d} \|_{L^{p}}$.
Similar arguments also give
$\| c \|_{\TLnaked^{\alpha}_{p, \infty}(A)} \asymp \| \tau \|_{\ell^{\infty}}$.
\\~\\
\textbf{Case 2.} Suppose $p = \infty$.
For $q < \infty$, it follows from the definition of $\| \cdot \|_{\TLnaked_{\infty, q}^{\alpha}(A)}$ that
\begin{align*}
 \| c \|_{\TLnaked_{\infty, q}^{\alpha}(A)}
 &\geq \bigg(
         \dashint_{Q^A_{0,0}}
           \sum_{j \in \Z, j \leq 0}
             |\det(A)|^{-j q (\alpha + \frac{1}{2})}
             \sum_{k \in \Z^d}
               |c_{j,k}|^q \Indicator_{Q^A_{j,k}} (x)
         \; dx
       \bigg)^{\frac{1}{q}} \\
 &= \bigg(
      \int_{[0,1]^d} \sum_{j \in \Z, j \leq 0} |\tau_{-j}|^{q} \sum_{k \in I_j} \Indicator_{Q^A_{j,k}} (x) \; dx
    \bigg)^{1/q} \\
 &= \bigg(
      \int_{[0,1]^d} \sum_{j \in \Z, j \leq 0} |\tau_{-j}|^{q}  \Indicator_{\Omega_j} (x) \; dx
    \bigg)^{1/q}
    = \bigg( \int_{[0,1]^d} \sum_{j \in \Z, j \leq 0} |\tau_{-j}|^{q}  \; dx \bigg)^{1/q} \\
 &= \| \tau \|_{\ell^{q}},
\end{align*}
where the penultimate step used that $[0, 1]^d \subseteq \Omega_j$ for $j \leq 0$.
On the other hand, using that $\Indicator_{\Omega_j} \leq \Indicator_{\mathcal{B}_{R'}(0)}$ for $j \leq 0$, yields
\begin{align*}
 \| c \|_{\TLnaked^{\alpha}_{\infty, q} (A)}
 &= \sup_{Q \in \mathcal{Q}^A}
    \bigg(
      \dashint_Q
      \bigg[
        \bigg(
          \sum_{j \in \Z, j \leq \scale_A(Q)}
            |\det(A)|^{-j q (\alpha + \frac{1}{2})}
            \sum_{k \in \Z^d}
              |c_{j,k}|^{q} \Indicator_{Q^A_{j,k}} (x)
        \bigg)^{1/{q}}
      \bigg]^{q} \; dx
    \bigg)^{1/{q}} \\
 &\leq \bigg\|
         \bigg(
           \sum_{j \in \Z}
             |\det(A)|^{-j q (\alpha + \frac{1}{2})}
             \sum_{k \in \Z^d}
               |c_{j,k}|^{q} \Indicator_{Q^A_{j,k}}
         \bigg)^{\frac{1}{q}}
       \bigg\|_{L^{\infty}} \\
 &= \bigg\| \bigg( \sum_{j \in \Z, j \leq 0} |\tau_{-j}|^{q} \Indicator_{\Omega_j} \bigg)^{1/q} \bigg\|_{L^{\infty}} \\
 &\leq \| \tau \|_{\ell^q} \| \Indicator_{\mathcal{B}_{R'} (0)} \|_{L^{\infty}},
\end{align*}
whenever $q < \infty$.
It is immediate that $\| c \|_{\TLnaked^{\alpha}_{\infty, \infty}(A)} = \| \tau \|_{\ell^{\infty}}$.
Thus, also $\| c \|_{\TLAi} \asymp \| \tau \|_{\ell^q}$ for all $q \in (0, \infty]$.

To complete the proof, note that \Cref{lem:basic_concidence} implies because of
$\TLnaked^{\alpha}_{p, q_1} (A) =  \TLnaked^{\alpha}_{p, q_2} (A)$
(cf.\ the beginning of the proof) that
\[
 \| \tau \|_{\ell^{q_1}}
 \asymp \| c \|_{\TLnaked^{\alpha}_{p, q_1} (A)}
 \asymp \| c \|_{\TLnaked^{\alpha}_{p, q_2} (A)}
 \asymp \| \tau \|_{\ell^{q_2}},
\]
with implicit constant independent of $\tau \in \CC^{N_0}$,
and hence $q_1 = q_2$.
\end{proof}

\begin{proof}[Proof of \Cref{thm:necessary_intro}]
 \Cref{thm:necessary_intro} follows from a combination of \Cref{prop:p1=p2}, \Cref{thm:p=q} and \Cref{prop:q1=q2}.
\end{proof}

\section{Application: A spectral characterization of isotropic Triebel-Lizorkin sequence spaces}
\label{sec:spectral}

This section provides a spectral characterization of those expansive matrices $A \in \mathrm{GL}(d, \mathbb{R})$
such that $\TL(A) = \TL(2 \cdot I_d)$ for all $p, q \in (0, \infty]$ and $\alpha \in \mathbb{R}$,
that is, those matrices generating the classical isotropic Triebel-Lizorkin sequence space $\TL(2 \cdot I_d)$.

We will use the following lemma on periodic matrices.
Recall that a matrix $A \in \mathrm{GL}(d, \mathbb{R})$ is called periodic
whenever $A^k = I_d$ for some $k \in \mathbb{N}$.
Although we expect this lemma to be part of the folklore, we provide a proof for the sake of completeness.

\begin{lemma} \label{lem:spectral}
Let $A \in \R^{d \times d}$.
Then the following assertions are equivalent:
\begin{enumerate}[label=(\roman*)]
  \item $A$ is periodic;
  \item $A$ is diagonalizable over $\CC$ and all eigenvalues of $A$ belong to the set
        \[\{ z \in \CC \,\,:\,\, \exists \, k \in \N: z^k = 1 \} ; \]
  \item There exists an invertible matrix $S \in \R^{d \times d}$ such that
        \[
          S^{-1} A S
          = \diag (B_1,\dots,B_b)
        \]
        is a block-diagonal matrix, where each block $B_j$ is either a $1 \times 1$
        matrix of the form $B_j = (\pm 1)$, or a $2 \times 2$ rotation matrix
        with a \enquote{rational angle}, i.e.,
        \[
          B_j
          = R_{\phi_j}
          = \begin{pmatrix}
            \cos (\phi_j) & - \sin(\phi_j) \\
            \sin(\phi_j) & \cos(\phi_j)
            \end{pmatrix}
          \qquad \text{with} \qquad
          \phi_j \in 2 \pi \Q
          .
        \]
\end{enumerate}
\end{lemma}

\begin{proof}
  (i) $\Rightarrow$ (ii):
  Let $k \in \N$ with $A^k = I_d$.
  Then, for the polynomial $p(X) = X^k - 1$, we have $p(A) = 0$,
  meaning the minimal polynomial of $A$ divides $p$.
  But $p$ has $k$ distinct zeros, namely $e^{2 \pi i j / k}, j = 0,\dots,k-1$.
  Hence, the minimal polynomial of $A$ factors into distinct linear factors over $\CC$.
  By \mbox{\cite[Chapter~6, Theorem~6]{HoffmanKunzeLinearAlgebra}},
  this means that $A$ is diagonalizable over $\CC$.
  Moreover, each eigenvalue $z \in \CC$ of $A$ is a zero of the minimal polynomial,
  and hence of $p$, and thus satisfies $z^k = 1$.
  \\~\\
 (ii) $\Rightarrow$ (iii):
  Write the eigenvalues of $A$ (repeated according to their multiplicity) as
  \[ \mu_1,\dots,\mu_{b_r}, \mu_{b_r + 1}, \overline{\mu_{b_r + 1}},\dots, \mu_b, \overline{\mu_b}, \]
  where $\mu_1,\dots,\mu_{b_r} \in \R$ and where $\Im (\mu_j) > 0$ for $b_r < j \leq b$.
  This is possible, since for a real matrix, the complex eigenvalues
  come in \enquote{conjugate pairs}.
  Since all eigenvalues of $A$ belong to
  $\{ z \in \CC \,\,:\,\, \exists \, k \in \N: z^k = 1 \}$,
   there exists $k \in \N$ such that all of the eigenvalues $\mu_j$
  are of the form $\mu_j = e^{2 \pi i t_j / k}$ for some $t_j \in \N_0$.
  This in particular implies $\mu_j \in \{ \pm 1 \}$ for $1 \leq j \leq b_r$.
  Since $A$ is diagonalizable over $\CC$,
  we can invoke \cite[Corollary~3.4.1.10]{HornJohnsonMatrixAnalysis}
  about the \enquote{real Jordan normal form} of diagonalizable matrices
  to conclude that there exists an invertible matrix $S \in \R^{d \times d}$ such that
  \[
    S^{-1} A S
    = \diag(B_1,\dots,B_b)
  \]
  is a block-diagonal matrix, where
  \[
    B_j
    = \begin{cases}
        (\mu_j) = (\pm 1),
        & \text{for } 1 \leq j \leq b_r, \\
        \left(\begin{smallmatrix} a_j & b_j \\ -b_j & a_j \end{smallmatrix}\right),
        & \text{for } b_r < j \leq m \text{ and } \mu_j = a_j + i b_j .
      \end{cases}
  \]
  Finally, note for $b_r < j \leq m$ that
  \[
    \begin{pmatrix}
      a_j & b_j
      \\ -b_j & a_j
    \end{pmatrix}
    = \begin{pmatrix}
      \cos(2 \pi t_j / k) & \sin(2 \pi t_j / k) \\
      - \sin(2 \pi t_j / k) & \cos(2 \pi t_j / k)
      \end{pmatrix}
    = R_{- 2 \pi t_j / k}
  \]
  is a rotation matrix with a \enquote{rational rotation angle}.
  \\~\\
  (iii) $\Rightarrow$ (i):
  It is well-known that the rotation matrices satisfy
  $R_\phi R_\theta = R_{\phi + \theta}$.
  Hence, for the case where $B_j = R_{\phi_j}$ with $\phi_j = 2 \pi \frac{t_j}{\ell_j}$
  with $t_j \in \Z$, $\ell_j \in \N$,
  we have \[ B_j^{\ell_j} = R_{\phi_j}^{\ell_j} = R_{\ell_j \phi_j} = R_{2 \pi t_j} = I_2. \]
  Similarly, if $B_j = (\pm 1)$, then for $\ell_j := 2$ we have $B_j^{\ell_j} = I_1$.
  Overall, this shows for $k := \ell_1 \cdots \ell_b$ that
  \begin{align*}
    A^k
    & = \bigl(S \diag(B_1,\dots,B_b) S^{-1}\bigr)^k \\
    & = S \diag(B_1^k,\dots,B_b^k) S^{-1} \\
    & = S I_d S^{-1}
    = I_d
    ,
  \end{align*}
  which shows that $A$ is periodic.
\end{proof}

The following theorem provides a spectral characterization of the matrices inducing
isotropic Triebel-Lizorkin sequence spaces.

\begin{theorem}
Let $A \in \mathrm{GL}(d, \mathbb{R})$. Then the following assertions are equivalent:
\begin{enumerate}
 \item[(i)] $\TL(A) = \TL(2 \cdot I_d)$ for all $\alpha \in \mathbb{R}$ and $p,q \in (0, \infty]$;
 \item[(ii)] There exists
             an invertible matrix $S \in \R^{d \times d}$ such that
             \[
               \frac{A}{2}
               = S \diag (B_1,\dots,B_b) S^{-1}
             \]
             is a block-diagonal matrix, where each $B_j$ is either a $1 \times 1$ matrix of the form
             $(\pm 1)$, or a $2 \times 2$ rotation matrix with a rotation angle in
             $2 \pi \Q$;
\item[(iii)] $A$ is diagonalizable over $\CC$
             and all eigenvalues of $A$ belong to the set
             \[
               \bigl\{ 2 z \in \CC \,\,:\,\, z \in \CC \text{ and } \exists \, k \in \N : z^k = 1 \bigr\}
               .
             \]
\end{enumerate}
\end{theorem}
\begin{proof}
 A combination of \Cref{thm:sufficient_intro}, \Cref{thm:necessary_intro} and \Cref{lem:characterization}
 shows that (i) is equivalent to $A^k = 2^k \cdot I_d$ for some $k \in \mathbb{N}$,
 i.e., $A/2$ is a periodic matrix.
 The equivalences follow therefore directly from \Cref{lem:spectral}.
\end{proof}

\appendix

\section{Postponed proof}

\begin{proof}[Proof of \Cref{lem:basic}]
(i) If $(c^{(n)})_{n \in \mathbb{N}}$ is a sequence consisting of elements $c^{(n)} \in \TLA$
with $\liminf_{n \to \infty} \| c^{(n)} \|_{\TLA} < \infty$ and $c \in \CC^{\Z \times \Z^d}$
satisfies $|c_{j,k}| \leq \liminf_{n \to \infty} |c_{j,k}^{(n)}|$ for all $j \in \Z$ and $k \in \Z^d$,
then it is easily verified via an application of Fatou's lemma that
$c \in \TLA$ with $\| c \|_{\TLA} \leq \liminf_{n \to \infty} \| c^{(n)} \|_{\TLA}$.
This means that the (quasi-)normed space $\TLA$ satisfies the Fatou property,
and hence it is complete by \cite[Section 65, Theorem 1]{zaanen1967integration};
see also \cite[Lemma 2.2.15]{VoigtlaenderPhDThesis} for the case of quasi-norms.
Moreover, it follows that $\TLA$ is solid, meaning that if $c \in \TLA$ and $c' \in \CC^{\Z \times \Z^d}$
are such that $|c'_{j,k}| \leq |c_{j,k}|$ for all $j \in \Z$ and $k \in \Z^d$,
then $c' \in \TLA$ and $\| c' \|_{\TLA} \leq \| c \|_{\TLA}$.
In turn, this implies the pointwise convergence of sequences converging in $\TLA$.
More precisely, $|c_{j,k} | \cdot \| e_{j,k} \|_{\TLA} =  \| c_{j,k} e_{j,k} \|_{\TLA} \leq \| c \|_{\TLA}$.
\\~\\
(ii) Fix $\varepsilon \in (0,1)$ and consider the (quasi-)norm $\| \cdot \|_{\TLA}^\ast$
defined in \Cref{lem:borelsets}.
Given $c \in \TLA$ and Borel sets $E_{j,k} \subseteq Q_{j,k}^A$
with $|E_{j,k}| / |Q^A_{j,k}| > \varepsilon$, we define
\[
  c^{E_{j,k}}_{j,k}(x) := |\det(A)|^{- j (\alpha + 1/2)} |c_{j,k}| \mathds{1}_{E_{j,k}}(x), \quad x \in \R^d.
\]
Then \Cref{lem:borelsets} yields
\[
  \| c \|_{\TLA}^\ast
  = \inf
    \bigg\{
      \bigg\|
        \Bigl\|
          \Bigl( c^{E_{j,k}}_{j,k} (\cdot) \Bigr)_{j \in \Z, k \in \Z^d}
        \Bigr\|_{\ell^q}
      \bigg \|_{L^p}
      :
      E_{j,k} \subseteq Q^A_{j,k}, \; \frac{|E_{j,k}|}{|Q^A_{j,k}|} > \varepsilon
    \bigg\}.
\]
Using that
\begin{align*}
  \bigg\|
    \Bigl\|
      \Bigl( c^{E_{j,k}}_{j,k} (\cdot) \Bigr)_{j \in \Z, k \in \Z^d}
    \Bigr\|_{\ell^q}
  \bigg \|_{L^p}
  = \bigg\|
      \Bigl\|
        \Bigl(c^{E_{j,k}}_{j,k} (\cdot) \Bigr)_{j \in \Z, k \in \Z^d}^r
      \Bigr\|_{\ell^{q/r}}^{1/r}
    \bigg \|_{L^p}
  = \bigg\|
      \Bigl\|
        \Bigl(c^{E_{j,k}}_{j,k} (\cdot) \Bigr)_{j \in \Z, k \in \Z^d}^r
      \Bigr\|_{\ell^{q/r}}
    \bigg \|^{1/r}_{L^{p/r}}
\end{align*}
and
\[
  \Bigl( (a + b)_{j,k}^{E_{j,k}} \Bigr)^r
  \leq \Bigl(a^{E_{j,k}}_{j,k} \Bigr)^r + \Bigl(b^{E_{j,k}}_{j,k} \Bigr)^r
\]
since $r \leq 1$, the claim follows easily from applications of the triangle inequality,
which is applicable since $q/r, p/r \geq 1$.
\end{proof}

\section*{Acknowledgments}
The authors thank M. Bownik for posing the question of classifying discrete Triebel-Lizorkin spaces.
For J.v.V., this research was funded in whole or in part
by the Austrian Science Fund (FWF): 10.55776/J4555 and 10.55776/PAT2545623.
For open access purposes, the author has applied a CC BY public copyright license
to any author-accepted manuscript version arising from this submission.
J.v.V. is grateful for the hospitality
of the Katholische Universität Eichstätt-Ingolstadt during his visit.
F.V. acknowledges support by the Hightech Agenda Bavaria.

\end{document}